\documentclass{article}

\usepackage{amsmath, amssymb, amsthm} 

\usepackage{xspace}
\allowdisplaybreaks
\expandafter\let\expandafter\oldproof\csname\string\proof\endcsname
\let\oldendproof\endproof
\renewenvironment{proof}[1][\proofname]{%
	\oldproof[\bf #1]%
}{\oldendproof}

\allowdisplaybreaks

\theoremstyle{plain}
\newtheorem{theorem}{Theorem}[section]
\newtheorem{lemma}[theorem]{Lemma}
\newtheorem{claim}[theorem]{Claim}
\newtheorem{proposition}[theorem]{Proposition}

\newtheorem{problem}[theorem]{Problem}

\newtheorem{definition}[theorem]{Definition}

\newtheorem{strategy}[theorem]{Strategy}

\newtheorem{setting}[theorem]{Setting}
\newtheorem{question}[theorem]{Question}

\newcommand{\Bin}{\ensuremath{\textrm{Bin}}}
\newcommand{\whp}{w.h.p.}

\newcommand{\Builder}{\textsf{Builder}\xspace}

\RequirePackage[normalem]{ulem} 
\RequirePackage{color}\definecolor{RED}{rgb}{1,0,0}\definecolor{BLUE}{rgb}{0,0,1} 

\title{Very fast construction of bounded-degree spanning graphs\\ via the semi-random graph process}

\author{
	Omri Ben-Eliezer \thanks{Blavatnik School of Computer Science, Raymond and Beverly Sackler Faculty of Exact Sciences, Tel Aviv University, Tel Aviv, 6997801, Israel. Email: omrib@mail.tau.ac.il.} 
	\and 
	Lior Gishboliner \thanks{School of Mathematical Sciences, Raymond and Beverly Sackler Faculty of Exact Sciences, Tel Aviv University, Tel Aviv, 6997801, Israel. Email: liorgis1@mail.tau.ac.il. Supported in part by ERC Starting Grant 633509.}
	\and 
	Dan Hefetz \thanks{Department of Computer Science, Ariel University, Ariel 40700, Israel. Email: danhe@ariel.ac.il. Research supported by ISF grant 822/18.}
	\and
	Michael Krivelevich \thanks{School of Mathematical Sciences, Raymond and Beverly Sackler Faculty of Exact Sciences, Tel Aviv University, Tel Aviv, 6997801, Israel. Email: krivelev@tauex.tau.ac.il. Partially supported by USA-Israel BSF grants 2014361 and 2018267, and by ISF grant 1261/17.}
}

\begin{document}
\begin{titlepage}

\clearpage\maketitle
\thispagestyle{empty}

\begin{abstract}
\emph{Semi-random processes} involve an adaptive decision-maker, whose goal is to achieve some predetermined objective in an online randomized environment. 
In this paper, we consider a recently proposed semi-random graph process, defined as follows: we start with an empty graph on $n$ vertices, and in each round, the decision-maker, called \Builder, receives a uniformly random vertex $v$, and must immediately (in an online manner) choose another vertex $u$, adding the edge $\{u,v\}$ to the graph. \Builder's end goal is to make the constructed graph satisfy some predetermined monotone graph property.
There are also natural offline and non-adaptive modifications of this setting.

We consider the property $\mathcal{P}_H$ of containing a spanning graph $H$ as a subgraph. It was asked by N. Alon whether for every bounded-degree $H$, \Builder can construct a graph satisfying $\mathcal{P}_H$ with high probability in $O(n)$ rounds. We answer this question positively in a strong sense, showing that any graph with maximum degree $\Delta$ can be constructed with high probability in $(3\Delta/2 + o(\Delta)) n$ rounds, where the $o(\Delta)$ term tends to zero as $\Delta \to \infty$. This is tight (even for the offline case) up to a multiplicative factor of $3 + o_{\Delta}(1)$. 
Furthermore, for the special case where $H$ is a forest of maximum degree $\Delta$, we show that $H$ can be constructed with high probability in $O(\log \Delta)n$ rounds. This is tight up to a multiplicative constant, even for the offline setting. Finally, we show a separation between {\em adaptive} and {\em non-adaptive} strategies, proving a lower bound of $\Omega(n\sqrt{\log n})$ on the number of rounds necessary to eliminate all isolated vertices w.h.p. using a non-adaptive strategy. This bound is tight, and in fact there are non-adaptive strategies for constructing a Hamilton cycle or a $K_r$-factor, which are successful w.h.p. within $O(n\sqrt{\log n})$ rounds. 
\end{abstract}
\end{titlepage}


\section{Introduction} \label{sec:intro}


Recently, the following \emph{semi-random graph process} was proposed by Peleg Michaeli, and analyzed by Ben-Eliezer, Hefetz, Kronenberg, Parczyk, Shikhelman, and Stojakovi{\'c} \cite{BHKPSS}. A single adaptive player, called \Builder, starts with an empty graph $G$ on a set $V$ of $n$ vertices. The process then proceeds in \emph{rounds}, where in each round \Builder is offered a uniformly random vertex $v$, and chooses an edge of the form $\{v,u\}$ to add to the graph $G$. 
\Builder's objective is typically to construct a graph that satisfies some predetermined monotone graph property; for example, to make $G$ an expander with certain parameters, or to have $G$ contain a Hamilton cycle. The natural question arising in this context is the following:
\begin{center}
\emph{Given a monotone graph property $\mathcal{P}$, how many rounds of the semi-random graph process are required for \Builder to construct (with high probability\footnote{With probability that tends to one as $n \to \infty$; abbreviated \whp~henceforth.}) a graph which satisfies $\mathcal{P}$?}
\end{center}

Semi-random problems of this type, involving both randomness and intelligent choices made by a ``decision-maker'', have been widely studied in the algorithmic literature.  
One of the first (and most famous) results on such processes, established by Azar et al.~\cite{ABKU2000}, concerns sequential allocation of $n$ balls into $n$ bins, where the goal is to minimize the number of balls in the fullest bin. It is well-known that if each ball is simply assigned to a bin uniformly at random, then w.h.p., the fullest bin will contain $\Theta(\ln n /  \ln \ln n)$ balls at the end of the process. However, as was shown in \cite{ABKU2000}, very limited ``intelligent intervention'' substantially improves the above bound: if, instead of the random assignment, for any ball we are given \emph{two (random) choices} of bins to pick from, then the trivial strategy of always choosing the least loaded bin out of the two offered, results w.h.p. in the maximum bin load dropping to $\Theta(\ln \ln n)$ -- an exponential improvement. 
This idea has inspired many subsequent theoretical and practical results in various contexts within computer science, see e.g.~\cite{FNP, PR, TRL} for a small sample of these. 

In a sense, semi-random processes can be viewed as settings where an \emph{online algorithm} aims to achieve a predetermined objective in a \emph{randomized environment}. As opposed to the ``standard'' setting where online algorithms are measured in terms of their worst-case performance,  in the semi-random setting the task is to design online algorithms that achieve their goal with high (or at least constant) probability, and require as few rounds as possible. Further discussion of related semi-random graph models, such as the so-called Achlioptas model, can be found in Section \ref{subsec:other_semi_random}.



In this paper we continue the investigation into the semi-random graph process. The first work on this topic \cite{BHKPSS} proved upper and lower bounds on the number of rounds required to w.h.p.~satisfy various properties of interest. Among the upper bounds were a $O(n^{1-\varepsilon})$ bound for the property of containing a copy of any fixed graph $H$ (here $\varepsilon$ depends on $H$), a $O(n)$ bound for containing a perfect matching or a Hamilton cycle, and a $O(\Delta n)$ bound for the property of having minimum degree $\Delta$, as well as for the property of $\Delta$-vertex-connectivity. 
These results prompted Noga Alon to ask whether it is the case that every given (spanning) graph of bounded maximum degree can be constructed w.h.p.~in $O(n)$ rounds in this model. Formally, define $\mathcal{P}_H$ as the property of containing an (unlabeled) copy of $H$, i.e., as the property that there exists an injection 
$\varphi \colon V(H) \to V(G)$ (where $G$ is the graph constructed by \Builder), 
so that $\{\varphi(i), \varphi(j)\} \in E(G)$ for every $\{i,j\} \in E(H)$. 
\begin{question}[Alon, Question 6.2 in \cite{BHKPSS}]
	\label{question:Alon}
Consider the semi-random graph process over $n$ vertices.
Is it true that for every graph $H$ on $n$ vertices with bounded maximum degree, there exists a strategy which enables \Builder to w.h.p. construct a copy of $H$ in $O(n)$ rounds?
\end{question}
\noindent
The main result of this paper gives a positive answer to the above question (see Theorem \ref{thm:main}). 

\paragraph{The offline setting}
Before addressing Question \ref{question:Alon}, let us consider the easier \emph{offline} setting, where \Builder is provided {\em in advance} with the full sequence of vertices offered throughout the process. In other words, in this setting \Builder does not need to make his decisions online, but can rather 
choose all edges at once after seeing the sequence of random vertices. 
%

As an example, consider the case where \Builder's goal is to construct a triangle-factor\footnote{For a graph $F$ and an integer $n$ divisible by $|V(F)|$, the $n$-vertex $F$-factor is the graph which consists of $n / |V(F)|$ vertex-disjoint copies of $F$.}. 
Observe that if at some point of the process, at least $\frac{2n}{3}$ different vertices 
have been offered, among which at least $\frac{n}{3}$ vertices have been offered at least twice, then \Builder can already construct a triangle-factor (in the offline setting). Indeed, \Builder simply partitions the vertices into triples $\{u,v,w\}$, where $u$ was offered at least twice and $v$ was offered at least once, and chooses the edges $\{u,v\}$ and $\{u,w\}$ at rounds when $u$ was offered, and the edge $\{v,w\}$ at a round when $v$ was offered.  
It is easy to check that $O(n)$ rounds suffice w.h.p. to have at least $\frac{2n}{3}$ different vertices offered, and moreover to have at least $\frac{n}{3}$ vertices offered at least twice. Thus, in the offline setting \Builder can construct a triangle-factor in $O(n)$ rounds. 
This $O(n)$ bound can in fact be generalized in a strong sense to \emph{any} bounded-degree target graph $H$. Proposition 4.1 in~\cite{BHKPSS} -- which presents necessary and sufficient general winning conditions for \Builder in the offline setting -- implies that \Builder wins the game as soon as the list of offered vertices allows the construction of a suitable \emph{orientation} of $H$. 
In Section \ref{subsec:initial} we show how this can be used to get the following general offline result.
\begin{proposition} \label{prop::offline}
	Let $\Delta,n > 0$ be integers, and let $H$ be an $n$-vertex graph of maximum degree $\Delta$. In the offline version of the semi-random process on $n$ vertices, \Builder has a strategy allowing him to construct a copy of $H$ in 
	$(\Delta/2 + o(\Delta))n$ rounds w.h.p.
\end{proposition} 
The $o(\Delta)$ term here is a function of $\Delta$ satisfying $o(\Delta)/\Delta \rightarrow 0$ as $\Delta \to \infty$.
Proposition \ref{prop::offline} substantially extends Theorem 1.9 in \cite{BHKPSS}, which showed a similar result for the property of having minimum degree at least $k$ (with an explicit dependence on $k$). 
Proposition \ref{prop::offline} is clearly optimal up to the $o(\Delta)$ term, since $\Delta$-regular graphs on $n$ vertices have exactly $\Delta n / 2$ edges (and hence trivially require at least this number of rounds). 
More interestingly, it turns out that the $o(\Delta)$-term is unavoidable; it follows from \cite[Theorem 1.9]{BHKPSS} that at least 
$(1/2 + \varepsilon_{\Delta})\Delta n$ rounds are required for \Builder to construct a graph of minimum degree at least $\Delta$, where $\varepsilon_{\Delta} > 0$ (and $\varepsilon_{\Delta} \rightarrow 0$ as $\Delta \rightarrow \infty$).

\paragraph{Main result: online strategy for constructing bounded-degree spanning graphs}
We now return to the more challenging \emph{online} setting, where \Builder is offered vertices one-by-one and must (irrevocably) decide which edge to add immediately after being offered a vertex.
Our main result in this paper, Theorem \ref{thm:main}, asserts that \Builder can construct any given bounded-degree spanning graph in $O(n)$ rounds w.h.p.

\begin{theorem} \label{thm:main}
	Let $\Delta, n > 0$ be integers and let $H$ be an $n$-vertex graph of maximum degree $\Delta$. In the online version of the semi-random graph process on $n$ vertices, \Builder has a strategy guaranteeing that \whp, after 
	$$
	\begin{cases}
	(\Delta/2 + o(\Delta)) n, & \textrm{ if } \Delta = \omega(\log n) \\
	(3\Delta/2 + o(\Delta)) n, & \textrm{ otherwise } 
	\end{cases}
	$$ 
	rounds of the process, the constructed graph will contain a copy of $H$.
\end{theorem}

As before, the $o(\Delta)$ term here is a function of $\Delta$ satisfying $o(\Delta)/\Delta \rightarrow 0$ as $\Delta \to \infty$. Note that $\Delta$ is allowed to depend on $n$ arbitrarily. 
Theorem \ref{thm:main} answers Question \ref{question:Alon} in a strong sense: not only can any bounded-degree graph be constructed \whp~in a linear number of rounds, but in fact, the dependence on the maximum degree is very modest. This result clearly illustrates the power of semi-random algorithms compared to their truly random counterparts; see the discussion below on the appearance of various spanning structures in the random graph process.

The notion of \emph{competitive ratio} \cite{BE1998} refers to the performance of an online algorithm compared to the best offline algorithm for the same problem. 
In view of the trivial $\Delta n / 2$ lower bound and Proposition \ref{prop::offline}, our algorithm is $(3 + o_{\Delta}(1))$-competitive for general $H$, where the $o_{\Delta}(1)$ term tends to zero as $\Delta$ tends to infinity.  
As an open question, it will be very interesting to determine the optimal competitive ratio of an online algorithm for this problem.
\begin{problem} \label{prob:tight_bound}
	In the online version of the semi-random process on $n$ vertices, is it true that for every $n$-vertex graph $H$ of maximum degree $\Delta$, \Builder has a strategy to construct a copy of $H$ w.h.p.~in 
	$(\Delta/2 + o(\Delta)) n$ \nolinebreak rounds?
\end{problem}



We end this discussion with two additional problems, which ask for tight bounds on the number of rounds required to construct some specific graphs $H$ of particular interest.

\begin{problem}[Tight bounds for constructing a Hamilton cycle]\label{prob:Hamilton_cycle}
	Is there a number $\alpha_{\mathrm{Ham}}$ such that in the online version of the semi-random process on $n$ vertices, \Builder can w.h.p. construct a Hamilton cycle in $(\alpha_{\mathrm{Ham}} + o(1))n$ rounds, but w.h.p. cannot accomplish this in $(\alpha_{\mathrm{Ham}} - o(1))n$ rounds? If so, what is the value of $\alpha_{\mathrm{Ham}}$?
\end{problem}

\begin{problem}[Tight bounds for constructing a $K_r$-factor]\label{prob:clique_factor}
	For each $r \geq 2$, is there a number $\alpha_r$ such that in the online version of the semi-random process on $n$ vertices, for $n$ which is divisible by $r$, \Builder can w.h.p. construct a $K_r$-factor in $(\alpha_r + o(1))n$ rounds, but w.h.p. cannot accomplish this in $(\alpha_r - o(1))n$ rounds? If so, what is the value of $\alpha_r$?
\end{problem}

%

We note that Problem \ref{prob:clique_factor} is open even for $r = 2$, in which case \Builder's goal is to construct a perfect matching. Both the perfect matching problem and the Hamilton cycle problem were already considered in \cite{BHKPSS}, where some bounds for these problems were obtained. 

\paragraph{Constructing a bounded-degree spanning forest}

Theorem~\ref{thm:main} establishes that the (typical) number of rounds needed to construct a general spanning graph $H$ of maximum degree $\Delta$ is $O(\Delta n)$. 
This is clearly tight for graphs whose average degree is $\Theta(\Delta)$. It is now natural to ask if we can break the $\Theta(\Delta n)$ barrier for graphs with a much smaller average degree, such as trees. 
The next result answers this question positively.

\begin{theorem}\label{thm:trees}
	Let $\Delta,n > 0$ be integers and let $T$ be an $n$-vertex forest of maximum degree $\Delta$. In the online version of the semi-random process on $n$ vertices, \Builder has a strategy guaranteeing that \whp, after $O(n\log \Delta)$ rounds of the process, the constructed graph will contain a copy of \nolinebreak $T$.  
\end{theorem}  

The next proposition shows that the dependence on $n$ and $\Delta$ in Theorem \ref{thm:trees} is tight even for the offline version of the semi-random process.
\begin{proposition} \label{prop::offlineTree}
	For every $\Delta \geq 1$ and for every $n \geq n_0(\Delta)$, there exists a forest $T$ with $n$ vertices and maximum degree $\Delta$ satisfying the following. In the offline version of the semi-random graph process, \whp~\Builder needs 
	$\Omega(n\log \Delta)$ rounds in order to construct a copy of $T$.   
\end{proposition}
\noindent
We conclude this section by proposing the following problem:
\begin{problem}\label{prob:degenerate}
	How many rounds are required to construct $n$-vertex $d$-degenerate graphs of maximum degree \nolinebreak $\Delta$?
\end{problem}
\noindent
An answer to Problem \ref{prob:degenerate} would generalize Theorem \ref{thm:trees}, as forests are exactly the $1$-degenerate graphs. 

\paragraph{Non-adaptive strategies}
It is fairly natural to inquire whether imposing the restriction of {\em non-adaptivity} handicaps Builder, and if so, to which extent exactly. Here, by non-adaptivity we mean that, in a sense, \Builder's choices are decided upon beforehand, and do not depend on the situation at any given round of the process. The precise definition that we use is as follows. A {\em non-adaptive strategy} consists of a family ${\cal L}$ of adjacency lists ${\cal L}=\{L^w: w \in [n]\}$, where for each $w \in [n]$, the list $L^w=(L^w(i): i=1,\ldots,n-1)$ is a permutation of $[n]\setminus \{w\}$. The lists ${\cal L}=\{L^w: w \in [n]\}$ are specified {\em in advance} (i.e., before the sequence of random vertices starts being exposed). Playing according to such a strategy means that during the vertex exposure process $w_1,w_2,\ldots$, if in a given round vertex $w$ appears for the $i$th time, $i\ge 1$, then \Builder is obliged to connect $w$ to the $i$th vertex $L^w(i)$ on its list. (To avoid ambiguities, let us assume that if $w$ has already been connected to $L^w(i)$, then \Builder simply skips his move. This assumption will not change much \nolinebreak in \nolinebreak our \nolinebreak analysis.) 


It turns out that the non-adaptivity assumption indeed hampers Builder --- it takes him typically $\Omega(n\sqrt{\log n})$ rounds to get rid of isolated vertices, as stated in Theorem \ref{prop::non-adaptive_isolated_vertices} below. This is in rather sharp contrast with the situation for general (i.e. adaptive) strategies. Indeed, it is easy to see that \Builder can construct a connected graph in $n-1$ rounds (with probability $1$); constructing a graph with no isolated vertices can be done w.h.p. even faster, in $(\ln 2 + o(1))n$ rounds (see \cite{BHKPSS}); and finally, Theorem \ref{thm:main} shows that in fact every bounded-degree graph can be constructed in $O(n)$ rounds. 
The three theorems below are for the online version of the semi-random graph process.

\begin{theorem}\label{prop::non-adaptive_isolated_vertices}
	In the semi-random process on $n$ vertices, any non-adaptive strategy requires w.h.p. $\Omega(n \sqrt{\log n})$ rounds to construct a graph in which none of the $n$ vertices is isolated. 
\end{theorem}

It turns out that the lower bound of Theorem \ref{prop::non-adaptive_isolated_vertices} is tight in a strong sense:  
by executing an appropriate non-adaptive strategy for $O(n \sqrt{\log n})$ rounds, \Builder can construct w.h.p. several important spanning structures, such as a Hamilton cycle and a $K_r$-factor (for fixed $r$). This is expressed in the following two theorems.

\begin{theorem}\label{prop::non-adaptive_Hamilton_cycle}
	In the semi-random process on $n$ vertices, there is a non-adaptive strategy allowing \Builder to construct a Hamilton cycle in $8n\sqrt{\log n}$ rounds w.h.p.   
\end{theorem}

\begin{theorem}\label{prop::non-adaptive_clique_factor}
	For every $r \geq 2$ there is $C = C(r)$ such that for every $n$ which is divisible by $r$, there is a non-adaptive strategy allowing \Builder to construct a $K_r$-factor in $C n \sqrt{\log n}$ rounds w.h.p.
\end{theorem}

We remark that while the above definition of non-adaptivity is deterministic in nature (in the sense that the lists $\{L^w : w \in [n]\}$ are predetermined), the proof of Theorem \ref{prop::non-adaptive_isolated_vertices} can be easily adapted (with the same asymptotic lower bound) to apply also to
``random non-adaptive strategies", i.e. strategies in which for every $w \in [n]$ and $i \geq 1$, the vertex which \Builder connects to $w$ at the $i$th time that $w$ is sampled is drawn from some predetermined probability distribution on $[n] \setminus \{w\}$.  


\paragraph{The situation in the (``purely"-)random graph process}
A common theme in our results is that introducing ``intelligent choices" into a random setting can allow for a dramatic improvement. (Perhaps the first appearance of this theme is in the aforementioned work of Azar et al. \cite{ABKU2000}.)
To illustrate this phenomenon in the setting of random graph processes, let us compare our Theorem \ref{thm:main} with the situation in which all edge-choices are made completely randomly. 
Recall that the {\em random graph process} $\tilde{G} = (G_m)_{m=0}^{N}$ is defined by choosing a random permutation $e_1,\dots,e_N$ of all $N = \binom{n}{2}$ edges of $K_n$, and letting $G_m$ be the graph whose edge-set is $\{e_1,\dots,e_m\}$. We refer the reader to \cite{FK,JLR} for an overview of this classical object, as well as the standard random graph models $G(n,m)$ and $G(n,p)$, which are mentioned below. Note that for each $0 \leq m \leq \binom{n}{2}$, the graph $G_m$ is distributed as the Erd\H{o}s-R\'enyi graph $G(n,m)$, i.e. as a random graph chosen uniformly among all graphs with $m$ edges and $n$ (labeled) vertices. 
 
The appearance of bounded-degree spanning graphs in the random graph process (or equivalently, in $G(n,m)$) has been thoroughly investigated\footnote{We note that the results surveyed here were actually proved for the binomial random graph $G(n,p)$, which is the graph obtained by selecting each of the $\binom{n}{2}$ edges of $K_n$ with probability $p$ and independently. It is well known --- see e.g. \cite[Section 1.4]{JLR} or \cite[Section 1.1]{FK} --- that $G(n,m)$ is closely related to $G(n,p)$ when $p = m/\binom{n}{2}$, allowing the transfer of results between the two models.}.
For starters, a standard first moment argument (see e.g. \cite{JLR}) shows that a copy of $K_{\Delta+1}$ only appears in the random graph process after roughly $n^{2-2/\Delta}$ rounds. 
Thus, obtaining even a single $K_{\Delta+1}$-copy (let alone a $K_{\Delta+1}$-factor) requires much more than a linear number of rounds. 
Determining the typical time of the appearance of a $K_{\Delta+1}$-factor turned out to be a difficult problem. 
Following a long line of research, this problem was settled by (a special case of) a celebrated result by Johansson, Kahn and Vu \cite{JKV}, which states that 
a $K_{\Delta+1}$-factor appears in the random graph process at around $m = n^{2-2/(\Delta+1)} (\log n)^{1 / \binom{\Delta+1}{2}}$. A more general discussion on the appearance of bounded-degree graphs $H$ other than clique-factors (and on the closely-related notion of \emph{universality}) can be found in the recent work of Ferber, Kronenberg, and Luh \cite{FKL}; see in particular Conjecture 1.5 there.

Similar superlinear lower bounds (on the number of edges required in order to typically contain a $K_{\Delta+1}$-factor, or even a single $K_{\Delta+1}$) are known or can be shown for various other random graph models, such as the random regular graph, or the model $G_{k\text{-}out}$ (where one connects each vertex to exactly $k$ other randomly chosen vertices, discarding repetitions).
Thus, while it was shown in \cite{BHKPSS} that the semi-random graph process can \emph{simulate} the random models $G(n, p)$, $G(n, m)$, and $G_{k\text{-}out}$, the above discussion indicates that these cannot help in solving Question \ref{question:Alon}, and we must utilize the power of the intelligent player, \Builder, in a more imaginative way.

A similar comparison can be made between Theorem \ref{thm:trees} and the emergence of (bounded-degree) spanning trees in the random graph process. 
It is well known that in the random graph process, w.h.p. the last isolated vertex disappears only at around $m = \frac{1}{2}n\log n$. Thus, a superlinear number of rounds is required in order to contain a spanning tree w.h.p. Again, we observe here that intelligent choices speed up the time required to reach the given goal: it takes $\Theta(n \log n)$ rounds for the random graph process to contain even a single spanning tree, whilst in our semi-random graph process, the number of rounds required to contain 
a prescribed bounded-degree spanning tree is only $O(n)$.   

It is worth mentioning that a recent breakthrough of Montgomery \cite{Montgomery}, which confirms a conjecture of Kahn \cite{KLW}, shows that for a fixed $\Delta$, w.h.p. all spanning trees with maximum degree $\Delta$ appear in $G_m$ after $m = C n \log n$ rounds (where $C = C(\Delta)$ is a large enough constant). 
We refer the reader to \cite{Montgomery} for further references to many other related works on this subject.

\subsection{Related Work : Semi-Random Processes}
\label{subsec:other_semi_random}
Perhaps the most studied semi-random graph process is now known as the \emph{Achlioptas process}, and was proposed by Dimitris Achlioptas in 2000. Similarly to our semi-random process, the Achlioptas process is a one-player game in which the player, \Builder, gradually constructs a graph.
The process runs in rounds, where in each round, two uniformly random edges are picked from the set of all $\binom{n}{2}$ edges of the $n$-vertex complete graph, or alternatively (depending on the version of the process) from all edges untaken at this point. These two edges are offered to \Builder, who then must choose exactly one of them and add it to the graph. 

While in our random graph process \Builder's goal is always to make his graph {\em satisfy} some given graph property, in the context of the Achlioptas process the goal is often to {\em avoid} satisfying a given property for as long as possible. In fact, Achlioptas's original question was whether \Builder can delay the appearance of a giant component beyond its typical time of appearance in the (``purely"-)random graph process. 
This question was answered positively by Bohman and Frieze \cite{BF}, see also \cite{Achlioptas2009,BK,GPS,riordan2012,SW}. Similar problems have also been studied for other properties or objectives; for example, the problem of avoiding a fixed subgraph \cite{KLohS,MST}, or the problem of speeding up the appearance of a Hamilton cycle \cite{KLS}. Achlioptas-like processes involving two choices were also investigated in other contexts, see e.g.~\cite{muller2015} for a geometric perspective. 

\subsection{Paper Organization and Notation} 
In Section~\ref{sec:preliminary} we state several auxiliary results.  Section \ref{sec:main} contains the proofs of Theorem \ref{thm:main} and Proposition~\ref{prop::offline}, as well as the description and analysis of Strategy~\ref{strategy}, which is the key tool used in our proofs. Theorem~\ref{thm:trees} and Proposition~\ref{prop::offlineTree} are proved in Section~\ref{sec:trees}. Finally, the proofs of Theorems \ref{prop::non-adaptive_isolated_vertices}, \ref{prop::non-adaptive_Hamilton_cycle} and \ref{prop::non-adaptive_clique_factor} appear in Section \ref{sec:non_adaptive}. 
Since the statements of Theorems~\ref{thm:main} and~\ref{thm:trees} are asymptotic (in both $n$ and $\Delta$), we always assume, where needed, that $n$ and $\Delta$ are sufficiently large. All logarithms are base $e$. We omit floor and ceiling signs whenever these are not crucial. 

\section{Preliminaries} \label{sec:preliminary}
We start by stating three known concentration inequalities that will be used in this paper. The first is a standard Chernoff-type bound (see, e.g.,~\cite{JLR}), the second is a simplified version of Azuma's inequality (see, e.g.,~\cite[Theorem 2.27]{JLR}), and the third is a simplified version of Talagrand's inequality (see, e.g.,~\cite{MR}).   

\begin{lemma} \label{lem:Chernoff}
	Let $X$ be a binomial random variable. Then, for every $\lambda \geq 0$, it holds that
	$$
	\mathbb{P}[X \leq \mathbb{E}[X] - \lambda] \leq e^{- \frac{\lambda^2}{2\mathbb{E}[X]}}
	$$
	and that
	$$
	\mathbb{P}[X \geq \mathbb{E}[X] + \lambda] \leq e^{- \frac{\lambda^2}{2 (\mathbb{E}[X] + \lambda/3)}} \; .
	$$
\end{lemma}
\begin{lemma}\label{lem:Azuma}{\cite[Theorem 2.27]{JLR}}
	Let $X$ be a non-negative random variable, not identically $0$, which is determined by $T$ independent trials $w_1, \ldots, w_T$. Suppose that $c \in \mathbb{R}$ is such that changing the outcome of any one of the trials can change the value of $X$ by at most $c$. Then, for every $\lambda \geq 0$, it holds that 
	$$
	\mathbb{P}\left[ X \leq \mathbb{E}[X] - \lambda \right] \leq e^{- \frac{\lambda^2}{2 c^2 T}}
	$$
	and that
	$$
	\mathbb{P}\left[ X \geq \mathbb{E}[X] + \lambda \right] \leq e^{- \frac{\lambda^2}{2 c^2 T}} .
	$$ 
\end{lemma}

\begin{lemma}\label{lem:Talagrand}{\cite[Pages 80-81]{MR}}
	Let $X$ be a non-negative random variable, not identically $0$, which is determined by $T$ independent trials $w_1, \ldots, w_T$.
	Suppose that $c,g > 0$ are such that 
	\begin{enumerate}
		\item Changing the outcome of any one of the trials can change the value of $X$ by at most $c$.
		\item For every $s$, if $X \geq s$ then there is a set of at most $g \cdot s$ trials whose outcomes certify\footnote{To be precise, this means that if $(w_1,\dots,w_T)$ is such that $X \geq s$, then there is $k \leq g \cdot s$ and $1 \leq i_1 < \dots < \nolinebreak i_k \leq T$ such that changing the outcome of any trials other than $w_{i_1},\dots,w_{i_k}$ does not change the fact that $X \geq s$.} that \nolinebreak $X \geq s$. 
	\end{enumerate}
	Then, for every $0 \leq \lambda \leq \mathbb{E}[X]$, it holds that
	$$
	\mathbb{P}\left[ |X - \mathbb{E}[X]| > \lambda + 60c\sqrt{g\mathbb{E}[X]} \right] 
	\leq 4e^{-\frac{\lambda^2}{8c^2g\mathbb{E}[X]}} \; \; .
	$$
\end{lemma}

\noindent
We will also need the following lemma regarding ``balanced" orientations of graphs. 

\begin{lemma} \label{lem::orientations}
	Let $H$ be an $n$-vertex graph of maximum degree $\Delta$. Then there exists an orientation $D$ of the edges of $H$ which satisfies the following two conditions:
	\begin{description}
		\item [(a)] $d_{D}^+(u) \leq \lfloor \Delta/2 \rfloor + 1$ for every $u \in V(H)$;
		\item [(b)] there exists a set $A \subseteq V(H)$ of size $|A| \geq \frac{n}{\Delta^2 + 1}$ such that $d_{D}^+(u) = 0$ for every $u \in A$. 
	\end{description}
\end{lemma}
\begin{proof} 
	Let $H^2$ denote the square of $H$, that is, the graph obtained from $H$ by adding an edge between any two vertices at distance $2$ in $H$. Let $A \subseteq V(H)$ be a maximum independent set in $H^2$; clearly $|A| \geq \frac{n}{\Delta^2 + 1}$. Let $H_0 = H \setminus A$. If $d_{H_0}(u)$ is even for every $u \in V(H_0)$, then let $H_1 = H_0$. Otherwise, let $H_1$ denote the graph obtained from $H_0$ by adding a new vertex $x$ and connecting it by an edge to every vertex of odd degree in $H_0$. Orient the edges of each connected component of $H_1$ along some Eulerian cycle of that component. Orient every edge of $E_H(V(H) \setminus A, A)$ from $V(H) \setminus A$ to $A$. Delete $x$ and denote the resulting oriented graph by $D$. Observe that $A$ is independent in $H$ and thus $D$ is an orientation of all the edges of $H$. It is evident that $D$ satisfies (b). Moreover, $D$ satisfies (a) as $A$ is independent in $H^2$ and thus $d_H(u, A) \leq 1$ for every $u \in V(H) \setminus A$.       
\end{proof}


\section{Constructing Spanning Graphs of Given Maximum Degree} \label{sec:main}
In this section we prove Theorem~\ref{thm:main} and Proposition~\ref{prop::offline}. The tools we develop here will also be used in the proof of Theorem~\ref{thm:trees} in Section~\ref{sec:trees}. 
We start by introducing some definitions and sketching a rough outline of the proof scheme that we will use to prove Theorem \ref{thm:main}. From this point onward, we fix an $n$-vertex graph $H$ with maximum degree $\Delta$. 
We assume that the ground-set of vertices for the semi-random process is $[n]$. The following definition will play an important role in our arguments.
\begin{definition}\label{def:good_set}
	Let $G$ be a graph on the vertex set $[n]$, and let $\varphi : V(H) \rightarrow [n]$ be a bijection. We say that a subset $A \subseteq V(H)$ is {\em $(G, \varphi)$-good} if $\varphi$ maps all edges of $H$ contained in $A$ to edges of $G$; that is, $A$ is $(G,\varphi)$-good if $\{\varphi(x), \varphi(y)\} \in E(G)$ for every $x, y \in A$ such that $\{x,y\} \in E(H)$. 
\end{definition}
The graph $G$ under consideration will always be \Builder's graph at some given moment during the process. It will usually be clear which moment we are considering, and so we will omit $G$ from the notation, simply writing ``$\varphi$-good". Note that, if at some point during the process, there is a bijection $\varphi : V(H) \rightarrow [n]$ such that $V(H)$ is $\varphi$-good, then \Builder has succeeded in constructing a copy of $H$. 


Our strategy consists of two stages: in the first stage, \Builder fixes an arbitrary bijection \linebreak $\varphi : V(H) \rightarrow [n]$ and plays so as to construct a $\varphi$-good set which is as large as possible. We show (see Lemma~\ref{lem:initial} below) that \Builder can w.h.p. guarantee the existence of a $\varphi$-good set which covers ``almost all" of the vertices of $H$, within $(\Delta/2 + o(\Delta))n$ rounds. This part of the argument is fairly straightforward, and is given in Section \ref{subsec:initial}.  In the second stage, which is far more involved and constitutes the heart of the proof, \Builder tries to iteratively extend this $\varphi$-good set by updating the embedding $\varphi$. We will show that by using a suitable ``role-switching" strategy (i.e., Strategy~\ref{strategy}), \Builder can ensure that w.h.p. $V(H)$ will be $\varphi$-good after $(\Delta + o(\Delta))n$ additional rounds. This is done in Lemma~\ref{lem:main}.


\subsection{The Initial Embedding, Proposition \ref{prop::offline}, and the Easy Case of Theorem \ref{thm:main}}\label{subsec:initial}

In this section we describe and analyze the first stage of \Builder's strategy. Along the way we prove Proposition \ref{prop::offline}, as well as the easy part of Theorem \ref{thm:main}, which corresponds to the regime $\Delta = \omega(\log n)$. 
In what follows, we will need the following lemma. 

\begin{lemma}\label{lem:number_of_appearances}
	Let $n,d$ be positive integers, let $\alpha \in (0,0.1)$, and suppose that 
	$d \gg \log(1/\alpha)$. 
	In the course of $(d + o(d))n$ rounds of the semi-random process on $n$ vertices, 
	the number of vertices $i \in [n]$ which were offered at most $d$ times is w.h.p. less than $\alpha n$.
\end{lemma}
\begin{proof}
	Set	$\ell := \big( d + \sqrt{6 d \log(1/\alpha)} \big)n = (d + o(d))n$ (here we use our assumption that $d \gg \log(1/\alpha)$), and 
	let $w_1,\dots,w_{\ell}$ be the first $\ell$ random vertices offered to \Builder. Let $U$ be the set of all vertices $i \in [n]$ such that 
	$\left| \{1 \leq j \leq \ell : w_j = i\} \right| \leq d$. Evidently, 
	$\left| \{1 \leq j \leq \ell : w_j \in U\} \right| \leq d \cdot |U|$. So in order to prove the lemma, it is enough to show that w.h.p., 
	every $U \subseteq [n]$ of size $\alpha n$ satisfies the inequality
	$X_U := \left| \{ 1 \leq j \leq \ell : w_j \in U \} \right| >
	d \cdot |U| = d\alpha n$. 
	
	Fix any $U \subseteq [n]$ of size $|U| = \alpha n$. Note that $X = X_U$ has the distribution $\Bin( \ell, \frac{|U|}{n} )$, and in particular
	$\mathbb{E}[X] = \ell|U|/n = \alpha \ell$. 
	By Lemma \ref{lem:Chernoff} with $\lambda = \alpha \ell - d \alpha n$, we have
	\begin{align*}
	\mathbb{P}\left[X < d \alpha n\right] &= 
	\mathbb{P}\left[X < \mathbb{E}[X] - \lambda\right] \leq
	\exp \left( {-\frac{\alpha^2n^2 \cdot \left( \ell/n - d\right)^2}{2\alpha\ell}} \right) 
	\\ &\leq 
	\exp \left( {-\frac{\alpha n \cdot \left( \ell/n - d\right)^2}{2d+o(d)}} \right) =
	e^{-(3 + o_d(1))\alpha \log(1/\alpha) n } \leq e^{-2.5\alpha\log(1/\alpha)n} \; \; ,
	\end{align*}
	where the last inequality holds if $d$ is large enough. 
	Now our assertion follows by taking the union bound \nolinebreak over \nolinebreak all 
	$$
	\binom{n}{\alpha n} \leq \left( e/\alpha \right)^{\alpha n} = 
	e^{(\log(1/\alpha) + 1) \cdot \alpha n} \leq  
	e^{2\alpha \log(1/\alpha)n}
	$$
	sets $U \subseteq [n]$ of size $|U| = \alpha n$.
\end{proof}

\begin{proof}[Proof of Proposition~\ref{prop::offline}]
	Since the statement of the proposition is asymptotic in $\Delta$, we may (and will) assume that $\Delta$ is large enough, where needed. Moreover, in light of Theorem \ref{thm:main}, we may assume that $\Delta = O(\log n)$ (because otherwise, \Builder can w.h.p. construct a copy of $H$ in $(\Delta/2 + o(\Delta))n$ rounds even in the online version of the semi-random process).
	Let $D$ be an orientation of the edges of $H$ as in Lemma~\ref{lem::orientations}. Let $w_1, w_2, \ldots$ denote the sequence of random vertices \Builder is offered. Let $m(D)$ denote the smallest integer $m$ for which there exists a bijection $\varphi : V(H) \rightarrow [n]$ such that $\varphi(u)$ appears in $(w_1, w_2, \ldots, w_m)$ at least $d_D^+(u)$ times for every $u \in V(H)$. It follows by Proposition 4.1 in~\cite{BHKPSS} that the number of rounds needed for \Builder to construct a copy of $H$ in the offline version of the semi-random process is at most $m(D)$. Hence, in order to complete the proof of this proposition, it suffices to show that w.h.p., in the course of 
	$m := \left( \Delta/2 + o(\Delta) \right) n$ rounds, at least $n - n/(\Delta^2 + 1)$ vertices are offered at least $\lfloor \Delta/2 \rfloor + 1$ times each. But this is just the statement of Lemma \ref{lem:number_of_appearances} with parameters $d = \lfloor \Delta/2 \rfloor$ and $\alpha = 1/(\Delta^2+1)$.
\end{proof}

\noindent 
In subsequent proofs, \Builder will employ the following simple (non-adaptive) strategy.

\begin{strategy}\label{strategy:initial}
	Let $H$ be an $n$-vertex graph of maximum degree $\Delta$. Fix an orientation $D$ of the edges of $H$ which satisfies property (a) of Lemma~\ref{lem::orientations}. Let $\varphi : V(H) \to [n]$ be an arbitrary bijection. At any given round of the process, having been offered a random vertex $w$, \Builder chooses an arbitrary vertex $u \in \varphi\left( N_D^+(\varphi^{-1}(w)) \right)$ which is not adjacent to $w$ in his current graph, and claims the edge $\{u, w\}$; if no such vertex $u$ exists, then \Builder claims an arbitrary edge incident with $w$. 
\end{strategy}

\noindent
In the following proposition we prove the easy part of Theorem \ref{thm:main}.
\begin{proposition}\label{prop:Delta >> logn}
	For every $\varepsilon \in (0,1)$ there exists an integer $C$ for which the following holds. Let $n$ and $\Delta = \Delta(n) \geq C \log n$ be positive integers and let $H$ be an $n$-vertex graph of maximum degree at most $\Delta$. Then, in the online version of the semi-random process on $n$ vertices, \Builder has a strategy guaranteeing that \whp, after $(1 + \varepsilon) \frac{\Delta n}{2}$ rounds of the process, his graph will contain a copy of $H$.
\end{proposition}

\begin{proof}
	\Builder executes Strategy \ref{strategy:initial} for 
	$\ell := (1 + \varepsilon) \frac{\Delta n}{2}$ rounds. 
	In the notation of Strategy \ref{strategy:initial}, it is evident that if for every $x \in V(H)$, the vertex $\varphi(x)$ is offered at least $d_D^+(x)$ times, then \Builder is successful in building a copy of $H$. 
	Therefore, in order to complete the proof of the proposition, it suffices to show that w.h.p., for every $1 \leq i \leq n$, the vertex $i$ is offered at least $\lfloor \Delta/2 \rfloor + 1$ times in the course of these $\ell$ rounds of the process.
	
	Fix any $1 \leq i \leq n$, and let $Z_i$ be the random variable counting the number of times $i$ is offered during the $\ell$ rounds. Then $Z_i \sim \Bin(\ell, 1/n)$, implying that $\mathbb{E}[Z_i] = \ell/n = (1 + \varepsilon) \Delta/2$. Applying Lemma~\ref{lem:Chernoff} with $\lambda = \varepsilon \Delta/2$, we obtain  
	\begin{equation*}
	\mathbb{P}[Z \leq \Delta/2] = 
	\mathbb{P}[Z \leq \mathbb{E}[Z] - \varepsilon\Delta/2] \leq
	\exp\left( {- \frac{(\varepsilon \Delta/2)^2}{(1 + \varepsilon)\Delta}} \right) \leq    
	e^{-\Delta\varepsilon^2/8} \leq 
	n^{- C \varepsilon^2/8} \leq 1/n^2
	\end{equation*}
	where the penultimate inequality holds since $\Delta \geq C \log n$, and the last inequality holds for $C \geq 16 \varepsilon^{-2}$. A union bound then implies that with probability at least $1 - \frac{1}{n} = 1 - o(1)$, every $1 \leq i \leq n$ was offered at least $\lfloor \Delta/2 \rfloor + 1$ times, as required.
\end{proof}
We now return to the main case of Theorem \ref{thm:main}, assuming henceforth that 
$\Delta = O(\log n)$. In the following lemma, we analyze the aforementioned first stage of \Builder's strategy. 
\begin{lemma}\label{lem:initial}
	Let $n,\Delta$ be positive integers, let $\alpha \in (0,0.1)$, and suppose that 
	$\Delta \gg \log(1/\alpha)$. 
	Let $H$ be an $n$-vertex graph of maximum degree $\Delta$ and let $\varphi : V(H) \rightarrow [n]$ be a bijection. Then, \Builder has a strategy guaranteeing that after 
	$\left( \Delta/2 + o(\Delta) \right) n$ rounds of the process, w.h.p. there will be a $\varphi$-good set $A \subseteq V(H)$ of size at least $(1 - \alpha) n$. 
\end{lemma} 

\begin{proof}
	\Builder executes Strategy \ref{strategy:initial} for 
	$\ell = 
	\left( \Delta/2 + o(\Delta) \right) n$ rounds. 
	Let $A'$ be the set of all vertices $x \in V(H)$ such that $\varphi(x)$ was offered at least $\lfloor \Delta/2 \rfloor + 1$ times in the course of the $\ell$ rounds. Apply Lemma \ref{lem:number_of_appearances} with $d = \lfloor \Delta/2 \rfloor$ and with $\frac{\alpha}{2\Delta}$ in place of $\alpha$, to conclude that w.h.p. we have 
	$|A'| \geq (1 - \frac{\alpha}{2\Delta})n$. 
	Here we use the assumption that $\Delta \gg \log(1/\alpha)$, which is necessary in order to apply Lemma \ref{lem:number_of_appearances} with the above parameters. 
	From now on we assume that $|A'| \geq (1 - \frac{\alpha}{2\Delta})n$ (which happens with high probability).
	
	Now, let $B = \{u \in V(H) : N_D^-(u, V(H) \setminus A') \neq \emptyset\}$, where $D$ is the orientation from Strategy \ref{strategy:initial}. It readily follows from the description of \Builder's strategy (namely, Strategy \ref{strategy:initial}) that if $x \in A' \setminus B$, then after $\ell$ rounds of the process, $\varphi(x)$ is adjacent in \Builder's graph to every vertex of $\varphi\left( N_H(x) \right)$. Thus, 
	$A := A' \setminus B$ is $\varphi$-good. Since the maximum degree of $H$ is $\Delta$, it follows that $|B| \leq \Delta \cdot |V(H) \setminus A'| \leq \frac{\alpha n}{2}$. We conclude that w.h.p. 
	$|A| \geq (1 - \alpha) n$, completing the proof. 
\end{proof}

\subsection{Improving the Embedding} \label{subsec:main}
	In this section we introduce and analyze \Builder's strategy for the second stage (see Strategy~\ref{strategy} below). 
	Our starting point is the state of \Builder's graph immediately after applying the strategy given by Lemma \ref{lem:initial}. (The value of the parameter $\alpha$, with which this lemma is applied, will be chosen later; the bijection $\varphi$ can be chosen arbitrarily.)
	Our goal is to iteratively update $\varphi$, so as to maintain a $\varphi$-good set which gradually increases in size until it equals $V(H)$. 
	
	Before delving into the details, let us illustrate the idea behind Strategy \ref{strategy} by considering the following ``toy" example: suppose that at some point during the process, \Builder has already managed to obtain a bijection $\varphi : V(H) \rightarrow [n]$ for which there is a $\varphi$-good set $A$ of size $n-1$. Let $b$ denote the unique element of $V(H) \setminus A$. The fact that $A$ is $\varphi$-good means that $\varphi(A)$ spans a copy of $H[A]$ in \Builder's graph. So in order to make $\varphi$ an embedding of $H$ into \Builder's graph, it remains to connect $\varphi(b)$ to all of the vertices in $\varphi(N_H(b))$. A naive way of doing this would be for \Builder to wait until $\varphi(b)$ will have been offered $d_H(b)$ times, and at each such time, to connect $\varphi(b)$ to a new vertex in $\varphi(N_H(b))$. This, however, will not work, since the probability that $\varphi(b)$ is offered (even once) in the course of $O(n)$ rounds does not tend to $1$. So instead, \Builder will try to find another vertex in $[n]$ to ``play" the role $\varphi(b)$, and to have $\varphi(b)$ play the role which was previously played by that other vertex. To this end, \Builder fixes (a large number of) vertices $a_1,\dots,a_m \in A$ which are not adjacent to $b$. 
	(We note that in order to make the strategy work, we need the additional assumption that the neighbourhoods of $a_1, \ldots, a_m$ are pairwise-disjoint, but the reader may ignore this issue at the moment.)
	Now \Builder acts as follows: each time a vertex of $\{\varphi(a_1), \ldots, \varphi(a_m)\}$ is offered, \Builder connects it to some new vertex of $\varphi(N_H(b))$; and each time a vertex of $\bigcup_{i=1}^m \varphi(N_H(a_i))$ is offered, \Builder connects it to $\varphi(b)$. 
	Now, if at some point there is an index $1 \leq i \leq m$ such that $\varphi(a_i)$ has already been offered at least $\Delta(H)$ times and every vertex in $\varphi(N_H(a_i))$ has already been offered at least once, then at this point $\varphi(a_i)$ is adjacent in \Builder's graph to every vertex of $\varphi(N_H(b))$, and $\varphi(b)$ is adjacent in \Builder's graph to every vertex of $\varphi(N_H(a_i))$. Hence, \Builder can now safely ``switch" the roles of $\varphi(a_i)$ and $\varphi(b)$. Formally, \Builder defines a new bijection $\varphi' : V(H) \rightarrow [n]$ by setting $\varphi'(b) = \varphi(a_i)$, $\varphi'(a_i) = \varphi(b)$, and $\varphi'(x) = \varphi(x)$ for every $x \in V(H) \setminus \{b, a_i\}$. Then $\varphi'$ is an embedding of $H$ into \Builder's graph. A key point of this method is that, since there are many ``candidates" for the role of $b$ (i.e. the vertices $a_1, \ldots, a_m$), it is very likely that one of them will indeed be chosen to be switched with $b$.    
	
	We now give a precise definition of the setting in which we will apply our ``role-switching" strategy. 
\begin{setting} \label{setting}
	We are given a graph\footnote{We think of $G$ as \Builder's graph immediately after employing the strategy whose existence is guaranteed by Lemma~\ref{lem:initial} for the number of rounds specified in that lemma.} $G$ on the vertex-set $[n]$, a bijection $\varphi : V(H) \rightarrow [n]$, and a $\varphi$-good set $A \subseteq V(H)$. We set $B = V(H) \setminus A$ and write $B = \{b_1, \ldots, b_r\}$.  
	We are also given an integer $m > 0$ and distinct vertices $a_{i,k} \in A$, where $1 \leq i \leq r$ and $1 \leq k \leq m$. Finally, we are given an integer\footnote{The reason for allowing flexibility in the choice of $d$ (as opposed to simply letting $d$ be the maximum degree of $H$), is that in one application (namely, Theorem~\ref{thm:trees}), we will be able to make sure that the degrees of the vertices 
	$b_i, a_{i,1}, \ldots, a_{i,m}$ ($1 \leq i \leq r$)
	are much smaller than $\Delta(H)$, which will be crucial for obtaining the desired bound.} 
	$d$ such that $d_H(x) \leq d$ for every $x \in \bigcup_{i=1}^r \{b_i, a_{i,1}, \ldots, a_{i,m}\}$. We assume that the following two properties are satisfied. 
	\begin{enumerate}
		\item There are no edges in $H$ between 
		$\left\{a_{i,k} : (i,k) \in [r] \times [m]\right\}$ and $B$. 
		\item The sets $\{a_{i,k}\} \cup N_H(a_{i,k})$ are pairwise-disjoint, where $(i,k)$ run over all pairs in $[r] \times [m]$. 
	\end{enumerate} 
\end{setting} 
%

Throughout the second stage of his strategy, \Builder maintains and updates sets $A_t \subseteq V(H)$ and bijections $\varphi_t : V(H) \rightarrow [n]$. Initially, $A_0 = A$ and $\varphi_0 = \varphi$. For every positive integer $t$, the pair $(A_t, \varphi_t)$ will be defined immediately after round $t$ of the second stage. We also set $B_t = V(H) \setminus A_t$ (so in particular, $B_0 = B$). 
Finally, we let $G_t$ denote \Builder's graph after exactly $t$ rounds of the second stage (so in particular, $G_0 = G$). 
We are now ready to describe \Builder's strategy for round $t$ of the second stage (for any integer $t \geq 1$). 

\begin{strategy} \label{strategy}
	Let $w_t \in [n]$ be the random vertex \Builder is offered at round $t$ of the second stage.
	\begin{enumerate}
		\item If some pair $(i,k) \in [r] \times [m]$ is such that 
		$w_t \in \varphi\left( \{a_{i,k}\} \cup N_H(a_{i,k}) \right)$ and $b_i \in B_{t-1}$, then do:
		\begin{enumerate}
			\item If $w_t \in \varphi(N_H(a_{i,k}))$, then claim the edge $\{w_t,\varphi(b_i)\}$.
			\item If $w_t = \varphi(a_{i,k})$, then choose an arbitrary vertex $u \in \varphi_{t-1} \left( N_H(b_i) \cap A_{t-1} \right)$ which is not adjacent to $\varphi(a_{i,k})$ in $G_{t-1}$, and claim the edge $\{\varphi(a_{i,k}), u\}$. 
			\item Check whether $\varphi(b_i)$ is adjacent in $G_t$ to every vertex of $\varphi(N_H(a_{i,k}))$ and, moreover, $\varphi(a_{i,k})$ is adjacent in $G_t$ to every vertex of $\varphi_{t-1} \left( N_H(b_i) \cap A_{t-1} \right)$. If so, then set $A_t = A_{t-1} \cup \{b_i\}$ (and hence $B_t = V(H) \setminus A_t = B_{t-1} \setminus \{b_i\}$), and 
			$$
			\varphi_t(x) = 
			\begin{cases}
			\varphi_{t-1}(b_i) & x = a_{i,k}, \\
			\varphi_{t-1}(a_{i,k}) & x = b_i, \\
			\varphi_{t-1}(x) & x \in V(H) \setminus \{a_{i,k}, b_i\}.
			\end{cases}
			$$
			\item 
			Otherwise (i.e., if the condition in Item 1(c) does not hold), set $A_t = A_{t-1}$ and $\varphi_t = \varphi_{t-1}$.
		\end{enumerate}
		\item Else (i.e., if there is no pair $(i,k) \in [r] \times [m]$ which satisfies the condition in Item 1), claim an arbitrary edge which is incident with $w_t$; this edge will not be considered as part of \Builder's graph in our analysis. Set $A_t = A_{t-1}$ and $\varphi_t = \varphi_{t-1}$.
	\end{enumerate}
\end{strategy}
 
The operation of defining $A_t$ and $\varphi_t$ as done in Item 1(c), is referred to as {\em switching $a_{i,k}$ and $b_i$}. This name stems from the fact that we swap the vertices which play the roles of $b_i$ and $a_{i,k}$ in our current partial embedding $\varphi_t$ of $H$ into $G_t$. Evidently, switching $a_{i,k}$ and $b_i$ does not change the role of any of the other $n-2$ vertices. 
Switching $a_{i,k}$ and $b_i$ is only done if, roughly speaking, the vertex currently playing the role of $b_i$ can play the role of $a_{i,k}$, and the vertex currently playing the role of $a_{i,k}$ can play the role of $b_i$; the exact condition for switching is stated in Item 1(c) above. Note that the only pairs of vertices which can be switched are of the form $(a_{i,k}, b_i)$ for some $1 \leq i \leq r$ and $1 \leq k \leq m$. 
In the following lemma we collect several simple facts regarding Strategy~\ref{strategy}.
\begin{lemma} \label{lem:strategy_basic_facts}
	Consider the execution of Strategy~\ref{strategy} for $\ell$ consecutive rounds, where $\ell$ is an arbitrary positive integer. Then the following statements hold.
	\begin{enumerate}
		\item If $a_{i,k}$ and $b_i$ were switched in round $t$, then none of the vertices $b_i, a_{i,1}, \ldots, a_{i,m}$ was switched at any other round. 
		\item If $x \in V(H)$ was not switched at any round, then $\varphi_s(x) = \varphi(x)$ for every $0 \leq s \leq \ell$. If $x$ was switched in round $t$, then $\varphi_s(x) = \varphi(x)$ for every $0 \leq s \leq t-1$ and $\varphi_s(x) = \varphi_t(x)$ for every $t \leq s \leq \ell$.  
		\item For every $x \in V(H) \setminus \left(\bigcup_{i=1}^r \{b_i, a_{i,1}, \ldots, a_{i,m}\} \right)$, we have $\varphi_s(x) = \varphi(x)$ for every $0 \leq s \leq \ell$. 
		\item Fix $1 \leq i \leq r$, and define $N_t := \varphi_{t-1}\left( N_H(b_i) \cap A_{t-1} \right)$ for every $1 \leq t \leq \ell$. Then, 
		$N_t \subseteq N_{t'}$ holds for every $1 \leq t < t' \leq \ell$.
	\end{enumerate}
\end{lemma}

\begin{proof}
	We start with Item 1. Switching $a_{i,k}$ and $b_i$ in round $t$ forces $b_i \in A_t$. It then follows by the description of Strategy~\ref{strategy} that $b_i \in A_s$ for every $t < s \leq \ell$, making the condition in Item 1 of Strategy~\ref{strategy} false for $b_i$ in each of the rounds $t+1,\dots,\ell$. Hence, if $a_{i,k}$ and $b_i$ were switched in round $t$, then none of the vertices $b_i, a_{i,1}, \ldots, a_{i,m}$ could have been switched in any subsequent round. Moreover, none of these vertices could have been switched in any round prior to round $t$ as this would have made $b_i$ ineligible for switching in round $t$. This proves Item 1.
	Item 2 can be easily proved by induction, using Item 1 and the definition of the functions $(\varphi_s : 0 \leq s \leq \ell)$ in Strategy~\ref{strategy}.
	Item 3 follows from Item 2 and the fact that only vertices in $\bigcup_{i=1}^r \{b_i, a_{i,1}, \ldots, a_{i,m}\}$ can be switched. 
	
	Let us prove Item 4. Fix $1 \leq i \leq r$ and $1 \leq t < t' \leq \ell$, and let $v \in N_t$ be an arbitrary vertex; we will prove that $v \in N_{t'}$. Set $x = \varphi_{t-1}^{-1}(v)$, and note that $x \in N_H(b_i) \cap A_{t-1} \subseteq N_H(b_i) \cap A_{t'-1}$. 
	We will show that $\varphi_{t'-1}(x) = \varphi_{t-1}(x) = v$, which would imply that $v \in \varphi_{t'-1}\left( N_H(b_i) \cap A_{t'-1} \right) = N_{t'}$, as required.
	Assume first that $x \in A$. It then follows by Item 1 of Setting~\ref{setting} that $x \notin \bigcup_{j=1}^r \{b_j, a_{j,1}, \ldots, a_{j,m}\}$. By Item 3 of Lemma~\ref{lem:strategy_basic_facts} we then have $\varphi_s(x) = \varphi(x)$ for every $0 \leq s \leq \ell$; in particular, $\varphi_{t'-1}(x) = \varphi_{t-1}(x) = v$, as claimed. Suppose now that $x \in B$, that is, $x = b_j$ for some $1 \leq j \neq i \leq r$. Since $b_j = x \in A_{t-1}$, the vertex $b_j$ must have been switched with some $a_{j,k}$ prior to round $t$. Now, Item 2 of Lemma~\ref{lem:strategy_basic_facts} implies that $\varphi_{t'-1}(x) = \varphi_{t-1}(x) = v$ in this case as well. 
\end{proof}

\noindent
The following lemma can be thought of as a proof of the ``correctness" of Strategy~\ref{strategy}. 
\begin{lemma}\label{lem:algo_correctness}
	For every non-negative integer $t$, the set $A_t$ is $(G_t, \varphi_t)$-good.
\end{lemma}

\begin{proof}
	The proof is by induction on $t$. The base case $t = 0$ is immediate from our assumption that $A$ is $\varphi$-good (see Setting \ref{setting}), and the fact that $A_0 = A$, $\varphi_0 = \varphi$ and $G_0 = G$. For the induction step, fix some $t \geq 1$ and suppose that the assertion of the lemma holds for $t-1$. 
	Consider the execution of Strategy~\ref{strategy} in round $t$. If either the condition in Item 1 or the condition in Item 1(c) does not hold, then there is nothing to prove, since in that case $A_t = A_{t-1}$ and $\varphi_t = \varphi_{t-1}$. Suppose then that both of these conditions hold, and let $(i,k) \in [r] \times [m]$ be the pair satisfying the condition in Item 1 of Strategy \ref{strategy}. That is, we assume that $a_{i,k}$ and $b_i$ were switched in round $t$. 
	
	We need to show that for every $x,y \in A_t$, if $\{x,y\} \in E(H)$, then $\{\varphi_t(x),\varphi_t(y)\} \in E(G_t)$. Hence, let $x,y \in A_t$ be such that $\{x,y\} \in E(H)$. Note that $\{x,y\} \neq \{a_{i,k}, b_i\}$, as $\{a_{i,k}, b_i\} \notin E(H)$ by Item 1 in Setting~\ref{setting}. If $x,y \notin \{a_{i,k}, b_i\}$ then we have $x,y \in A_{t-1}$, $\varphi_t(x) = \varphi_{t-1}(x)$ and $\varphi_t(y) = \varphi_{t-1}(y)$; so our assertion that 
	$\{\varphi_t(x),\varphi_t(y)\} \in E(G_t)$ follows from the induction hypothesis for $t-1$. Therefore, without loss of generality, we may assume that $x \in \{a_{i,k}, b_i\}$ and $y \notin \{a_{i,k}, b_i\}$. This assumption implies that $\varphi_t(y) = \varphi_{t-1}(y)$ and that $y \in A_{t-1}$.  
	
	Suppose first that $x = a_{i,k}$. Since $\{x,y\} \in E(H)$, it follows that $y \in N_H(x) = N_H(a_{i,k})$. Now Items 1 and 2 of Setting~\ref{setting} imply that $y \notin \bigcup_{i=1}^r \{b_i, a_{i,1}, \ldots, a_{i,m}\}$, which in turn implies that $\varphi_t(y) = \varphi(y) \in \varphi(N_H(a_{i,k}))$, see Item 3 of Lemma~\ref{lem:strategy_basic_facts}. 
	As for $x$, it follows from the definition of $\varphi_t$ in Item 1(c) of Strategy~\ref{strategy} that $\varphi_t(x) = \varphi_t(a_{i,k}) = \varphi_{t-1}(b_i) = \varphi(b_i)$, where the last equality holds by Item 2 of Lemma~\ref{lem:strategy_basic_facts}. Since $a_{i,k}$ and $b_i$ were switched in round $t$, it follows by Item 1(c) of Strategy~\ref{strategy} that $\varphi_t(x) = \varphi(b_i)$ is adjacent in $G_t$ to all vertices of $\varphi(N_H(a_{i,k}))$. In particular, $\{\varphi_t(x), \varphi_t(y)\} \in E(G_t)$ as required. 
	
	Suppose now that $x = b_i$. Since $\{x,y\} \in E(H)$, it follows that $y \in N_H(b_i)$. Therefore, 
	$\varphi_t(y) = \varphi_{t-1}(y) \in \varphi_{t-1}(N_H(b_i) \cap A_{t-1})$. Observe that $\varphi_t(x) = \varphi_t(b_i) = \varphi_{t-1}(a_{i,k}) = \varphi(a_{i,k})$, where the last equality holds by Item 2 of Lemma~\ref{lem:strategy_basic_facts}. Since $a_{i,k}$ and $b_i$ were switched in round $t$, it follows by Item 1(c) of Strategy~\ref{strategy} that $\varphi(a_{i,k})$ is adjacent in $G_t$ to all vertices of $\varphi_{t-1}(N_H(b_i) \cap A_{t-1})$. In particular, $\{\varphi_t(x), \varphi_t(y)\} \in E(G_t)$. This concludes the proof of the lemma.
\end{proof}

In the following three lemmas, we consider the execution of Strategy~\ref{strategy} for $\ell$ rounds for some positive integer $\ell$. For every $1 \leq i \leq r$ and $1 \leq k \leq m$, we denote by $\mathcal{A}_{i,k}$ the event: ``$\varphi(a_{i,k})$ was offered at least $d$ times {\bf after} each of the vertices in $\varphi(N_H(a_{i,k}))$ had already been offered''. In other words, $\mathcal{A}_{i,k}$ is the event that there are indices $1 \leq t_1 < \ldots < t_q < s_1 < \ldots < s_d \leq \ell$, where $q = |N_H(a_{i,k})|$, such that each element of $\varphi(N_H(a_{i,k}))$ was offered in one of the rounds $t_1, \ldots, t_q$, and $\varphi(a_{i,k})$ was offered in each of the rounds $s_1, \ldots, s_d$.   

\begin{lemma} \label{lem:suff_condition_for_switch}
Let $1 \leq i \leq r$. If there exists some $1 \leq k \leq m$ for which $\mathcal{A}_{i,k}$ occurred, then $b_i \in A_{\ell}$. 
\end{lemma}

\begin{proof}
	Suppose for a contradiction that $\mathcal{A}_{i,k}$ occurred for some $1 \leq i \leq r$ and $1 \leq k \leq m$, but $b_i \notin A_{\ell}$, i.e., $b_i \in B_{\ell}$. This means that $b_i$ was not switched at any of the $\ell$ rounds for which we execute Strategy~\ref{strategy}. Set $q = |N_H(a_{i,k})|$ and let $1 \leq t_1 < \ldots < t_q < s_1 < \ldots < s_d \leq \ell$ be the round numbers appearing in the definition of $\mathcal{A}_{i,k}$. Item 1(a) of Strategy~\ref{strategy} dictates that whenever a vertex from $\varphi(N_H(a_{i,k}))$ is sampled, \Builder connects it to $\varphi(b_i)$. This implies that, for every $t_q \leq t \leq \ell$, every vertex of $\varphi(N_H(a_{i,k}))$ is adjacent in $G_t$ to $\varphi(b_i)$.
	
As in Lemma~\ref{lem:strategy_basic_facts}, we let $N_t = \varphi_{t-1} \left( N_H(b_i) \cap A_{t-1} \right)$ for each $1 \leq t \leq \ell$. Suppose first that there exists some $1 \leq j \leq d$ such that $\varphi(a_{i,k})$ is adjacent in $G_{s_j}$ to every vertex of $N_{s_j}$. Then by Item 1(c) of Strategy~\ref{strategy}, \Builder would have switched $a_{i,k}$ and $b_i$ in round $s_j$, which would contradict our assumption that $b_i$ was never switched. Hence, for every $1 \leq j \leq d$ there exists a vertex of $N_{s_j}$ which is not adjacent in $G_{s_j}$ to $\varphi(a_{i,k})$. It follows by Item 1(b) of Strategy~\ref{strategy} that, in round $s_j$, \Builder claims an edge $\{\varphi(a_{i,k}), u_j\}$ for some $u_j \in N_{s_j}$ which is not adjacent to $\varphi(a_{i,k})$ in $G_{s_j-1}$. Note that $u_1, \ldots, u_d$ are distinct. 
It follows by Item 4 of Lemma~\ref{lem:strategy_basic_facts} that $u_1, \ldots, u_d \in N_{s_d}$. On the other hand, $|N_{s_d}| \leq |N_H(b_i)| \leq d$, implying that $N_{s_d} = \{u_1, \ldots, u_d\}$. But this means that in the graph $G_{s_d}$, the vertex $\varphi(a_{i,k})$ is adjacent to every vertex of $N_{s_d}$, contrary to the above.  
\end{proof}

\noindent
The following technical lemma provides lower bounds on the probability of the events $\mathcal{A}_{i,k}$.

\begin{lemma} \label{lem:prob_switch_event}
	Fix any $1 \leq i \leq r$ and $1 \leq k \leq m$. 
	\begin{description}
		\item [(a)] If $\ell \geq 2d$, then
		\begin{equation*} \label{eq:prob_lower_bound}
		\mathbb{P}[\mathcal{A}_{i,k}] \geq \left( \frac{\ell^2}{12 d n^2} \right)^{d} \cdot e^{-\frac{(d+1)(\ell - 2d)}{n-d-1}}\; .
		\end{equation*}
		\item [(b)] If 
		$\ell \geq (\log (2d) + d + 3\sqrt{d})n$, then $\mathbb{P}[\mathcal{A}_{i,k}] \geq \frac{1}{4}$. 
	\end{description}
\end{lemma}
\begin{proof} 
	We start with Item (a). Since $|N_H(a_{i,k})| \leq d$, it follows that
\begin{align*}
	\mathbb{P}[\mathcal{A}_{i,k}] &\geq \binom{\ell}{2d} \cdot d! \cdot \left( \frac{1}{n} \right)^{2d} \cdot \left( 1 - \frac{d+1}{n} \right)^{\ell - 2d} 
	\geq \left( \frac{\ell}{2dn} \right)^{2d} \cdot \left( \frac{d}{e} \right)^d \cdot e^{-\frac{(d+1)(\ell - 2d)}{n-d-1}}  
	\geq \left( \frac{\ell^2}{12dn^2} \right)^{d} \cdot e^{-\frac{(d+1)(\ell - 2d)}{n-d-1}} \; ,
	\end{align*}
	where in the second inequality we used the estimates
	$1 - x \geq e^{-\frac{x}{1-x}}$ (which holds for every $0 < x < 1$) and 
	$d! \geq \left( \frac{d}{e} \right)^d$ (which holds for every $d \geq 1$).
	
	Next, we prove (b). Let $\mathcal{E}_k$ be the event that every vertex of $\varphi(N_H(a_{i,k}))$ was offered in the course of the first $\log (2d)n$ rounds, and let $\mathcal{F}_k$ be the event that $\varphi(a_{i,k})$ was offered at least $d$ times in the course of the last $(d + 3\sqrt{d})n$ rounds. Since 
	$\ell \geq (\log (2d) + d + 3\sqrt{d})n$
	by assumption, the events $\mathcal{E}_k$ and $\mathcal{F}_k$ are independent. Note that $\mathcal{E}_k \cap \mathcal{F}_k \subseteq \mathcal{A}_{i,k}$; that is, if both $\mathcal{E}_k$ and $\mathcal{F}_k$ occur, then so does $\mathcal{A}_{i,k}$. Therefore, $\mathbb{P}[\mathcal{A}_{i,k}] \geq \mathbb{P}[\mathcal{E}_k \cap \mathcal{F}_k] = \mathbb{P}[\mathcal{E}_k] \cdot \mathbb{P}[\mathcal{F}_k]$. The probability that $\mathcal{E}_k$ did not occur is at most
	$$
	d \cdot \left(1 - \frac{1}{n} \right)^{\log(2d) n} \leq d \cdot e^{- \log(2d)} = \frac{1}{2},
	$$
	and the probability that $\mathcal{F}_k$ did not occur equals the probability that $\Bin\left( \left(d + 3 \sqrt{d}\right)n, 1/n \right)$ is smaller than $d$, which is at most
	\begin{align*}
	\mathbb{P}\left[\Bin\left( \left(d + 3\sqrt{d} \right)n, 1/n \right) < d \right] 
	&\leq e^{- \frac{9d}{2(d + 3\sqrt{d})}}
	\leq e^{-\frac{9d}{8d}} \leq \frac{1}{2}\; .
	\end{align*}
	Here in the first inequality we used Lemma~\ref{lem:Chernoff} with $\lambda = 3\sqrt{d}$. 
	We thus conclude that $\mathbb{P}[\mathcal{A}_{i,k}] \geq 1/2 \cdot 1/2 = 1/4$ as claimed. 
\end{proof}



The following lemma forms the main result of Section \ref{subsec:main}, and plays a key role in the proofs of Theorems \ref{thm:main} and \ref{thm:trees}. Roughly speaking, this lemma states that if $\varphi : V(H) \rightarrow [n]$ is a bijection admitting a $\varphi$-good set that misses only a small fraction of $V(H)$, then by following Strategy~\ref{strategy} for a suitable (and not too large) number of rounds, \Builder can obtain a bijection $\varphi' : V(H) \rightarrow [n]$ which admits a $\varphi'$-good set that misses significantly fewer vertices. 
The actual statement is somewhat convoluted, since we simultaneously handle two regimes: one where the number of bad vertices is small but not very small, and another where this number is tiny (these also correspond to the two items in Lemma \ref{lem:prob_switch_event}). 
The proof of Lemma \ref{lem:iteration} utilizes Lemmas \ref{lem:algo_correctness}, \ref{lem:suff_condition_for_switch} and \ref{lem:prob_switch_event}, as well as some of the concentration inequalities from Section \ref{sec:preliminary}.  

\begin{lemma}\label{lem:iteration}
	Let $H$ be an $n$-vertex graph, let $\varphi : V(H) \rightarrow [n]$ be a bijection, let $V(H) = A \cup B$ be a partition, and consider a moment in the semi-random process at which $A$ is 
	$\varphi$-good with respect to \Builder's graph. Write $B = \{b_1, \ldots, b_r\}$. Let $m$ be a positive integer, and let $\{a_{i,k} \in A : 1 \leq i \leq r \textrm{ and } 1 \leq k \leq m\}$ be vertices which satisfy Items 1 and 2 of Setting~\ref{setting}. 
	Let $d \geq 100$ be such that $d_H(x) \leq d$ for every $x \in \bigcup_{i=1}^r \{b_i, a_{i,1}, \ldots, a_{i,m}\}$. 
	Let 
	$\ell_1 = (\log (2d) + d + 3\sqrt{d})n$,
	$\ell_2 = \lceil n \cdot m^{-1/4d} \rceil$, $q_1 = \frac{1}{4}$, and $q_2 = d^{-2d}m^{-1/2}$. Fix any $j \in \{1,2\}$, and suppose that $m q_j \geq 10^6 d$. Let 
	$$
	p = 
	\begin{cases}
	o(1), & \frac{m q_j}{64d} \geq \log n, \\
	e^{- \frac{\sqrt{n}}{O(d)}} \; , & \frac{m q_j}{64d} < \log n.
	\end{cases}
	$$
	Suppose that \Builder executes Strategy~\ref{strategy} for $\ell_j$ additional rounds. 
	Then, with probability at least $1-p$, after $\ell_j$ rounds of the process \Builder's graph $G$ will satisfy the following: there will be a bijection $\varphi' : V(H) \rightarrow V(G) = [n]$ and a partition $V(H) = A' \cup B'$ such that $A'$ is $(G, \varphi')$-good, $A \subseteq A'$, and 
	\begin{equation} \label{eq:main}
	|B'| \leq 
	\begin{cases}
	0, & \frac{m q_j}{64d} \geq \log n, \\
	5n \cdot e^{- \frac{m q_j}{256 d}} \; , & \frac{m q_j}{64d} < \log n.
	\end{cases}
	\end{equation}
\end{lemma}

\begin{proof}
We will show that the assertion of the lemma holds with $\varphi' = \varphi_{\ell_j}$, $A' = A_{\ell_j}$, and $B' = B_{\ell_j}$. The fact that $A_{\ell_j}$ is $\varphi_{\ell_j}$-good follows immediately from Lemma~\ref{lem:algo_correctness}. It thus remains to prove that \eqref{eq:main} holds for $B' = B_{\ell_j}$ with probability at least $1 - p$. 

Fix arbitrary indices $1 \leq i \leq r$ and $1 \leq k \leq m$. For $j \in \{1, 2\}$, let us denote by $\mathbb{P}_j(\mathcal{A}_{i,k})$ the probability that $\mathcal{A}_{i,k}$ occurred in the course of the first $\ell_j$ rounds of the process. We claim that $\mathbb{P}_j[\mathcal{A}_{i,k}] \geq q_j$. 
The fact that  $\mathbb{P}_1[\mathcal{A}_{i,k}] \geq \frac{1}{4} = q_1$ follows immediately from Item (b) of Lemma~\ref{lem:prob_switch_event} and our choice of $\ell_1$. 
As for $j=2$, recall that $\ell_2 = \lceil n \cdot m^{-1/4d} \rceil$, which implies that $\ell_2 \geq 2d$ holds for $n$ which is sufficiently large with respect to $d$ (since we trivially have
$m \leq n$, and since $d \geq 100$, something like $n \geq (2d)^{1.01}$ would suffice). Therefore, Item (a) of Lemma~\ref{lem:prob_switch_event} yields 
$$
\mathbb{P}_2[\mathcal{A}_{i,k}] 
\geq \left( \frac{\ell_2^2}{12 d n^2} \right)^{d} \cdot 
e^{- \frac{(d+1)(\ell_2 - 2d)}{n-d-1}} 
\geq \left( \frac{1}{12d} \right)^{d} \cdot m^{-1/2} \cdot e^{-d-1} \geq
d^{-2d} m^{-1/2} = q_2\; \;,
$$
where the last inequality holds for sufficiently large $d$ (our assumption that $d \geq 100$ suffices). 

We have thus proved our assertion that $\mathbb{P}_j[\mathcal{A}_{i,k}] \geq q_j$ holds for every $j \in \{1,2\}$. For the remainder of the proof, we fix an arbitrary $j \in \{1,2\}$ and suppose (as in the statement of the lemma) that $mq_j \geq 10^6d$. For convenience, we put $\ell := \ell_j$ and $q := q_j$. For every $1 \leq i \leq r$, let $X_i$ be the random variable counting the number of indices $1 \leq k \leq m$ for which $\mathcal{A}_{i,k}$ occurred in the course of the first $\ell$ rounds. It follows by linearity of expectation that 
$\mathbb{E}[X_i] \geq m q$.

	Now, consider the sequence $(w_1, \ldots, w_{\ell})$ of random vertices offered to \Builder, and observe that changing any one coordinate in this sequence can change the value of $X_i$ by at most $1$ (here we use Item 2 in Setting \ref{setting}). Furthermore, for every $s$, if $X_i \geq s$, then there is a set of at most $2 d s$ coordinates in the sequence $(w_1, \ldots, w_{\ell})$ which certify that $X_i \geq s$ (indeed, each event $\mathcal{A}_{i,k}$ that occurred is certified by a set of at most $2d$ coordinates). It thus follows by Lemma~\ref{lem:Talagrand} with $c = 1$, $g = 2d$, and $\lambda = \frac{\mathbb{E}[X_i]}{2}$, that
	$$
	\mathbb{P}\left[X_i < \frac{\mathbb{E}[X_i]}{2} - 60 \sqrt{2d \mathbb{E}[X_i]} \right] 
	\leq 4 e^{- \frac{(\mathbb{E}[X_i]/2)^2}{16 d \mathbb{E}[X_i]}} 
	= 4 e^{- \frac{\mathbb{E}[X_i]}{64 d}} \leq 4 e^{- \frac{mq}{64d}} \; \; .
	$$
	Therefore, with probability at least 
	$1 - 4 e^{- \frac{mq}{64d}}$, it holds that
	\begin{align*}
	X_i &\geq \frac{\mathbb{E}[X_i]}{2} - 60 \sqrt{2d \mathbb{E}[X_i]}  
	= \sqrt{\mathbb{E}[X_i]} \cdot \left( \frac{\sqrt{\mathbb{E}[X_i]}}{2} - 60 \sqrt{2d} \right)  
	\geq \sqrt{\mathbb{E}[X_i]} \cdot \left( \frac{\sqrt{m q}}{2} - 60 \sqrt{2d} \right)  
	> 0,
	\end{align*}
	where the last inequality follows from our assumption that $m q \geq 10^6 d$. 
	
	Now let $\mathcal{I}$ be the set of all $1 \leq i \leq r$ such that $X_i = 0$. It follows by Lemma \ref{lem:suff_condition_for_switch} that if some $1 \leq i \leq r$ satisfies $X_i > 0$, then $b_i \notin B_{\ell}$. Hence, we have 
	$B_{\ell} \subseteq \{b_i : i \in \mathcal{I}\}$. So to complete the proof it is enough to show that the bounds in \eqref{eq:main} hold for the set $\mathcal{I}$.  
	We have seen that
	$
	\mathbb{P}[i \in \mathcal{I}] \leq 4 e^{- \frac{mq}{64d}}
	$ 
	holds for every \nolinebreak $1 \leq i \leq r$. \nolinebreak Hence,
	$$
	\mathbb{E}[|\mathcal{I}|] \leq r \cdot 4 e^{- \frac{mq}{64d}} 
	\; . 
	$$
	Suppose first that $\frac{m q}{64d} \geq \log n$, and note that we have 
	$r \leq \frac{n}{m} \leq \frac{n}{\log n}$. It follows that
	$$
	\mathbb{E}[|\mathcal{I}|] \leq r \cdot 4 e^{- \frac{mq}{64d}} \leq
	\frac{n}{\log n} \cdot \frac{4}{n} = o(1).
	$$
	So by Markov's inequality, we have $|\mathcal{I}| = 0$ w.h.p., as required. 
	
	Suppose now that $\frac{mq}{64d} < \log n$. Observe that changing any one coordinate in the sequence $(w_1, \ldots, w_{\ell})$ of random vertices, can change the value of $|\mathcal{I}|$ by at most $1$. Hence, it follows by Lemma~\ref{lem:Azuma} with $c = 1$ and 
	$\lambda = n \cdot e^{- \frac{m q}{256 d}} \geq n^{3/4}$, that
	\begin{align*}
	\mathbb{P}\left[ |\mathcal{I}| \geq \mathbb{E}[|\mathcal{I}|] + \lambda \right] 
	\leq e^{- \frac{\lambda^2}{2 \ell}} \leq 
	e^{- \frac{n^{3/2}}{2 \ell}} \leq 
	e^{- \frac{\sqrt{n}}{O(d)}} \; ,
	\end{align*}
	where the last inequality holds since $\ell_1, \ell_2 = O(dn)$. We conclude that with probability at least
	$1 - e^{- \frac{\sqrt{n}}{O(d)}}$, we have 
	$
	|\mathcal{I}| \leq \mathbb{E}[|\mathcal{I}|] + \lambda \leq 
	\left( 4r + n \right) \cdot e^{- \frac{m q}{256 d}} \leq
	5n \cdot e^{- \frac{m q}{256 d}} 
	,
	$
	as required.
\end{proof}

\subsection{Putting it All Together} \label{subsec:thm_proof}
\noindent

In this section we iterate Lemma \ref{lem:iteration} to prove Lemma \ref{lem:main}, from which Theorem \ref{thm:main} then easily follows. Lemma \ref{lem:main} roughly states that given a bijection $\varphi : V(H) \rightarrow [n]$ which admits a $\varphi$-good set covering almost all of $V(H)$, \Builder can use $\varphi$ as a basis for constructing a copy of $H$ in his graph, and he can achieve this objective w.h.p. fairly quickly.
In the proof we will need the following simple claim, which asserts that we can satisfy the conditions listed in Setting~\ref{setting} with a relatively large choice of $m$. 

\begin{claim} \label{claim:choice_of_a_{i,k}}
	Let $H$ be an $n$-vertex $D$-degenerate graph of maximum degree $\Delta$. Let $A \cup B$ be a partition of $V(H)$ and suppose that $|A| \geq 4\Delta|B|$. Set $r := |B|$ and
	$$m := \left\lceil \frac{|A|}{8\Delta^2|B|} \right\rceil.$$ Then, there exist vertices
	$\left( a_{i,k} : 1 \leq i \leq r \textrm{ and } 1 \leq k \leq m \right)$ such that $a_{i,k} \in A$ and $d_H(a_{i,k}) \leq 2D$ for every $(i,k) \in [r] \times [m]$, and such that Conditions 1 and 2 of Setting~\ref{setting} are satisfied. 
\end{claim}

\begin{proof}
Let $A' := \{a \in A: |N_H(a) \cap A| \leq 2D \text{ and } N_H(a) \cap B = \emptyset\}$. Note that every $a \in A'$ satisfies $d_H(a) = |N_H(a) \cap A| \leq 2D$. Since $H$ is $D$-degenerate, there are less than $|A|/2$ vertices $a \in A$ which satisfy $|N_H(a) \cap A| > 2D$. As $\Delta(H) = \Delta$, there are at most $\Delta|B| \leq |A|/4$ vertices $a \in A$ which have a neighbour in $B$. Altogether, we get $|A'| \geq |A|/4$. 
Using again the assumption $\Delta(H) = \Delta$, we infer that there exists \nolinebreak an \nolinebreak integer 
$$
M \geq \frac{|A'|}{\Delta^2 + 1} \geq \frac{|A|}{8\Delta^2}
$$
and vertices $a_1, \ldots, a_M \in A'$ such that $\text{dist}_H(a_i, a_j) \geq 3$ for every $1 \leq i < j \leq M$. Recalling our choice of $m$, we index (a subset of) the vertices $a_1, \ldots, a_M$ by pairs $(i,k) \in [r] \times [m]$. 
For every $(i,k) \in [r] \times [m]$, let $a_{i,k}$ be the vertex in 
$\{a_1, \ldots, a_M\}$ which is indexed by the pair $(i,k)$.
Now Condition 1 of Setting~\ref{setting} is satisfied due to our choice of $A'$,
and Condition 2 of Setting~\ref{setting} is satisfied because $\text{dist}_H(a_{i,k}, a_{i',k'}) \geq 3$ for every choice of distinct pairs $(i,k), (i',k') \in [r] \times [m]$.	
\end{proof}

\begin{lemma}\label{lem:main}
	Let $H$ be an $n$-vertex $D$-degenerate graph of maximum degree $\Delta$.
	Set $d = \min\{2D,\Delta\}$. Suppose that at some point in the semi-random process, there is a bijection $\varphi : V(H) \rightarrow [n]$ and a partition 
	$V(H) = A \cup B$ such that $A$ is 
	$\varphi$-good with respect to \Builder's graph; such that $|B| \leq 10^{-8}\Delta^{-5}n$; and such that $d_H(b) \leq d$ for every $b \in B$. Then \Builder has a strategy guaranteeing that w.h.p., after $(d + o(d))n$ additional rounds his graph will contain a copy of $H$. 
\end{lemma}
\begin{proof}
	We may and will assume that $d$ is large enough, say $d \geq 100$. 
	\Builder's strategy consists of two phases, which correspond to the two cases (i.e. $j = 1$ and $j = 2$) in Lemma \ref{lem:iteration}. 
	\begin{description}
%
%
		\item [Phase 1:] 
		Let 
		$$
		m_0 := \left\lceil \frac{n}{8\Delta^2|B|} \right\rceil \; .
		$$
		Find vertices $(a_{i,k} : 1 \leq i \leq |B| \textrm{ and } 1 \leq k \leq m_0)$ such that $a_{i,k} \in A$ and
		$d_H(a_{i,k}) \leq d$ for every $(i,k) \in [r] \times [m_0]$, and such that
		Conditions 1 and 2 of Setting~\ref{setting} are satisfied (with $m = m_0$). Apply the strategy whose existence is guaranteed by Lemma~\ref{lem:iteration} with parameter $j=1$ and with $m = m_0$. 
	 Lemma~\ref{lem:iteration} (with $j=1$) ensures that after 
		$\ell_1 = (\log (2d) + d + 3\sqrt{d})n = (d + o(d))n$ rounds, there will be a bijection $\varphi_0 : V(H) \rightarrow [n]$ and a partition $V(H) = A_0 \cup B_0$ such that $A_0$ is $\varphi_0$-good with respect to \Builder's graph, 
		$A \subseteq A_0$ (and hence $B_0 \subseteq B$), and 
		\begin{equation}\label{eq:zeroth_iteration}
		|B_0| \leq 
		\begin{cases}
		0, & \frac{m_0}{256d} \geq \log n, \\
		5 n \cdot e^{- \frac{m_0}{1024d}} \; , & \frac{m_0}{256d} < \log n.
		\end{cases}
		\end{equation}
		If $B_0 = \emptyset$ then $A_0 = V(H)$ is $\varphi_0$-good, implying that \Builder has successfully embedded $H$ into his graph, and so \Builder is done. Otherwise, proceed to Phase 2. 
		
		\item [Phase 2:] 
		Define a sequence of bijections $\varphi_1, \varphi_2, \ldots$ from $V(H)$ to $[n]$, and a sequence of partitions $A_1 \cup B_1, A_2 \cup B_2, \ldots$ of $V(H)$, by performing the following steps for every integer $t \geq 1$ for which $B_{t-1} \neq \emptyset$. 
		\begin{description}
		\item [(a)] 
		Find vertices $\big( a_{i,k} : 1 \leq i \leq |B_{t-1}| \textrm{ and } 1 \leq k \leq m_{t} := \lceil \frac{n}{8\Delta^2|B_{t-1}|} \rceil \big)$ such that $a_{i,k} \in A_{t-1}$ and $d_H(a_{i,k}) \leq d$ for every $(i,k) \in [r] \times [m_{t}]$, and such that
		Conditions 1 and 2 of Setting~\ref{setting} are satisfied for the partition $A_{t-1} \cup B_{t-1}$ with $m = m_{t}$.
		
		\item [(b)] Invoke the strategy whose existence is guaranteed by Lemma~\ref{lem:iteration} with $j=2$, with $m = m_t$, and with input $\varphi_{t-1}$ and $A_{t-1} \cup B_{t-1}$. Using this strategy, \Builder obtains a bijection $\varphi_t : V(H) \rightarrow [n]$ and a partition $A_t \cup B_t$ of $V(H)$ such that $A_t$ is $\varphi_t$-good with respect to \Builder's graph, $A_{t-1} \subseteq A_t$ (and hence $B_t \subseteq B_{t-1}$), and 
		\begin{equation}\label{eq:iteration}
		|B_t| \leq 
		\begin{cases}
		0, & m'_{t} \geq \log n, \\
		5n \cdot e^{- m'_{t}/4} \; , 
		& 
		m'_{t} < \log n,
		\end{cases}	
		\end{equation}
		where 
		$$
		m'_{t} := \frac{m_{t} \cdot q_2}{64d} = 
		\frac{\sqrt{m_{t}}}{64d^{2d+1}} \; ,
		$$ 
		and $q_2$ is as in Lemma~\ref{lem:iteration}. (Here and later on we slightly abuse notation by hiding the fact that $q_2 = d^{-2d}/\sqrt{m_{t}}$ depends on $t$.)
		\end{description}
		\end{description}
		
		Having described \Builder's strategy, we now turn to prove that w.h.p. \Builder can follow it. 
		First, note that by Claim~\ref{claim:choice_of_a_{i,k}}, there exist vertices $(a_{i,k} : 1 \leq i \leq |B| \textrm{ and } 1 \leq k \leq m_0)$ such that $a_{i,k} \in A$ and $d_H(a_{i,k}) \leq d$ for every $(i,k) \in [r] \times [m_0]$, and 
		such that Conditions 1 and 2 of Setting~\ref{setting} hold. 
		Moreover, the conditions required for the application of Lemma~\ref{lem:iteration} with $j=1$ are satisfied as 
		$m_0 q_1 = \frac{m_0}{4} \geq \frac{n}{32 \Delta^2 |B|} > 10^6 \Delta \geq 10^6 d$, where the second inequality follows from our assumption that $|B| \leq 10^{-8} \Delta^{-5} n$. This shows that \Builder can follow Phase 1 of his strategy. It remains to show that w.h.p. \Builder can follow Phase 2 of his strategy. Similarly to Phase 1, the existence of the desired vertices 
		$(a_{i,k} \in A_{t-1} : 1 \leq i \leq |B_{t-1}| \textrm{ and } 1 \leq k \leq m_{t})$ for every given integer $t \geq 1$ follows from Claim~\ref{claim:choice_of_a_{i,k}} with input $A_{t-1} \cup B_{t-1}$.
		It remains to prove that the conditions of Lemma~\ref{lem:iteration} are met whenever \Builder wishes to apply it (with $j=2$). The fact that $A_{t-1}$ is $\varphi_{t-1}$-good for every positive integer $t$ is guaranteed by the previous applications of Lemma \ref{lem:iteration}. The fact that $d_H(b) \leq d$ for every $b \in B_{t-1}$ follows from our assumption that the same holds for every $b \in B$, and the fact that $B_0 \subseteq B$ and $B_i \subseteq B_{i-1}$ for every $i \geq 1$. 
		We now show that $\sqrt{m_{t}} \cdot d^{-2d} = m_{t} \cdot q_2 \geq 10^6 d$ holds for every integer $t \geq 1$ for which $B_{t-1} \neq \emptyset$. To this end, first note that
		\begin{equation}\label{eq:beta_0}
		|B_0| \leq 
		5n \cdot e^{- \frac{m_0}{1024 d}} \leq
		5n \cdot \exp \left( - \frac{n}{2^{13} \Delta^3 |B|} \right) 
		e^{} \leq 
		n \cdot e^{- \Delta^2}, \;
		\end{equation}
		where the first inequality follows from \eqref{eq:zeroth_iteration}, and the third from our assumption that
		$|B| \leq 10^{-8}\Delta^{-5}n$. Since $B_i \subseteq B_{i-1}$ for every $i \geq 1$, we have $|B_i| \leq n \cdot e^{- \Delta^2}$ for each $i \geq 0$. Now we obtain
		\begin{equation}\label{eq:vanishing_set_of_bad_vertices}
		m_{t} \cdot q_2 = 
		\sqrt{m_{t}} \cdot d^{-2d} \geq 
		\frac{\sqrt{n}}{3\Delta^{2\Delta+1}\sqrt{|B_{t-1}}|} 
		\geq 
		\frac{e^{\Delta^2/2}}{3\Delta^{2\Delta+1}} \geq 10^6 \Delta \geq 10^6 d,
		\end{equation}
		where the penultimate inequality holds for sufficiently large $\Delta$ (say, $\Delta \geq 100$). This shows that we can indeed apply Lemma~\ref{lem:iteration} with $j=2$ and with input $A_{t-1} \cup B_{t-1}$ for every integer $t \geq 1$ for which $B_{t-1} \neq \emptyset$. We conclude that \Builder can follow Phase 2 of his strategy.
		
		Before moving on to prove the correctness of \Builder's strategy, we first prove the following claim. 
\begin{claim} \label{cl::relative_size_of_set_of_bad_vertices}
Let $\zeta = e^{-\Delta^2}$ and, for every non-negative integer $t$, let $\beta_t = |B_t|/n$. Suppose either that $t = 0$ or that $t \geq 1$ and $m'_{t} < \log n$. Then
$\beta_t \leq \zeta^{t+1}$. 
\end{claim}
		
\begin{proof}
		Our proof proceeds by induction on $t$. The base case $t = 0$ follows from \eqref{eq:beta_0}.
		Let then $t \geq 1$ and suppose that $m'_{t} < \log n$. Observe that the sequence $m'_s$ is monotone non-decreasing in $s$ (this follows from the fact that $B_i \subseteq B_{i-1}$ for every $i \geq 1$). So either $t - 1 = 0$, or $m'_{t-1} < \log n$. In either case we can apply the induction hypothesis to get $\beta_{t-1} \leq \zeta^t$. Note that  
		\begin{equation}\label{eq:m'}
		m'_{t} = 
		\frac{\sqrt{m_{t}}}{64 d^{2d+1}} \geq 
		\frac{\sqrt{n}}{200 \Delta^{2\Delta+2} \sqrt{|B_{t-1}|}}
		\; \; .
		\end{equation}
		Now we get
		\begin{align*}
		\beta_t &\leq 
		5 \cdot e^{- m'_{t}/4} \leq 
		5 \cdot \exp\left( {- \frac{1}{800 \Delta^{2\Delta+2} \sqrt{\beta_{t-1}}} } \right) \leq 
		5 \cdot \exp\left( {- \frac{1}{800 \Delta^{2\Delta+2}\zeta^{t/2}} } \right)
		\\ &= 
		5 \cdot \exp\left( - \frac{e^{\Delta^2t/2}}{800 \Delta^{2\Delta+2}} \right) \leq 
		e^{- \Delta^2 (t+1)} = \zeta^{t+1} \; ,
		\end{align*}
		where the first inequality holds by \eqref{eq:iteration}, the second inequality holds by~\eqref{eq:m'}, the third inequality holds by the induction hypothesis for $t-1$, and the last inequality holds for every $t \geq 1$, provided that $\Delta$ is larger than some suitable absolute constant (again, $\Delta \geq 100$ suffices). This proves the claim.
\end{proof}		

Returning to the proof of the lemma, we now prove the correctness of \Builder's strategy. For the time being, we will assume that all applications of 
Lemma~\ref{lem:iteration} throughout \Builder's strategy are successful; later we will show that w.h.p. this is indeed the case. It follows from 
\eqref{eq:m'} and from Claim \ref{cl::relative_size_of_set_of_bad_vertices} that 
$m'_t \geq e^{\Delta^2t/2} \cdot \frac{1}{200\Delta^{2\Delta+2}}$ holds for every $t \geq 1$. Hence,  $m'_{t} \geq \log n$ must hold for some $t \leq \log \log n$ (and in fact much earlier, but we will not need this). Now, if $m'_{t} \geq \log n$ then by \eqref{eq:iteration} we have $B_t = \emptyset$, which in turn implies that \Builder has successfully embedded $H$ into the graph he is constructing.  


	Next, we estimate the probability that \Builder's strategy fails. Recall that  Lemma~\ref{lem:iteration} is only applied once with parameter $j=1$, and that this application is w.h.p.~successful. Let us now consider the applications of Lemma~\ref{lem:iteration} with $j=2$ (in Phase 2). As previously noted, there is at most one such application with $m'_{t} \geq \log n$, and at most $\log \log n$ such applications with $m'_{t} < \log n$. The failure probability of the former application is $o(1)$, and the failure probability of each of the latter applications is at most 
	$e^{- \frac{\sqrt{n}}{O(d)}}$. We thus conclude that w.h.p. all of the above applications of Lemma~\ref{lem:iteration} are successful, as required. This completes the proof of correctness of \Builder's strategy. 
	
	It remains to estimate the overall number of rounds required for implementing \Builder's strategy. Recall that the sole application of Lemma~\ref{lem:iteration} with $j=1$ requires $(d + o(d))n$ rounds. It thus remains to bound from above the number of rounds required for Phase 2 of \Builder's strategy. To this end, let $t^*$ denote the smallest integer $t$ satisfying $m'_{t} \geq \log n$, and note that $t^* \leq \log\log n$. Then in Phase 2, Lemma~\ref{lem:iteration} was invoked at most $t^*$ times. Moreover, for each $1 \leq t \leq t^*$, invoking Lemma~\ref{lem:iteration} with input $A_{t-1} \cup B_{t-1}$ (and with $j=2$) required at most
	$$
	\big\lceil n \cdot m_{t}^{-\frac{1}{4d}} \big\rceil \leq n \cdot m_{t}^{-\frac{1}{4\Delta}} + 1 \leq
	n \cdot \left( \frac{n}{8\Delta^2|B_{t-1}|} \right)^{-\frac{1}{4\Delta}} + 1 \leq 
	O(n) \cdot \left( \frac{1}{\beta_{t-1}} \right)^{-\frac{1}{4\Delta}} + 1 \leq 
	O(n) \cdot e^{- \Delta t/4} + 1 
	$$
	rounds, where in the last inequality we used Claim \ref{cl::relative_size_of_set_of_bad_vertices}.
	Therefore, the overall number of rounds required for the (at most) $t^*$ applications of Lemma~\ref{lem:iteration} in Phase 2 is no more than 
	\begin{align*}
	\sum_{t=1}^{t^*} {\left( O(n) \cdot e^{- \Delta t/4} + 1 \right)} \leq 
	O(n) \cdot \sum_{t=1}^{\infty} {e^{- \Delta t/4}} + \log \log n = O(n).
	\end{align*}
	We conclude that the overall number of rounds required for implementing \Builder's strategy is at most \linebreak$(d + o(d)) n$, thus completing the proof of Lemma \ref{lem:main}. 
\end{proof}

\noindent
Equipped with Lemma \ref{lem:main}, we can finally prove Theorem \ref{thm:main}.
\begin{proof}[Proof of Theorem \ref{thm:main}]
	Let $n$, $\Delta$, and $H$ be as in the statement of the theorem. Due to Proposition \ref{prop:Delta >> logn}, we only need to handle the case $\Delta = O(\log n)$. Set $\alpha = 10^{-8}\Delta^{-5}$. 
	\Builder's strategy for embedding $H$ is as follows. Fix an arbitrary bijection $\varphi : V(H) \rightarrow [n]$. In the first
	$
	(\Delta/2 + o(\Delta))n
	$
	rounds, \Builder invokes the strategy whose existence is guaranteed by Lemma~\ref{lem:initial}, and thus obtains w.h.p. a $\varphi$-good set $A \subseteq V(H)$ of size at least $(1 - \alpha)n$. Setting $B := V(H) \setminus A$, observe that the requirements of Lemma \ref{lem:main} are satisfied for $D := \Delta$. \Builder now applies the strategy given by Lemma \ref{lem:main} for an additional $(\Delta + o(\Delta))n$ rounds, and by doing so successfully constructs a copy of $H$ in his graph w.h.p. The overall number of rounds is then $(3\Delta/2 + o(\Delta))n$, as required. 
\end{proof}

\section{Constructing Spanning Forests}\label{sec:trees}
In this section we prove Theorem~\ref{thm:trees} and Proposition~\ref{prop::offlineTree}.  
We start with the following simple lemma, whose proof demonstrates a strategy for greedily embedding an almost-spanning forest.
\begin{lemma} \label{lem:greedy_embedding}
	Let $n$ be a positive integer and let $\alpha \in (0,1)$ be such that $n \gg \alpha^{-2} \log(1/\alpha)$. Let $T'$ be a forest on $(1 - \alpha) n$ vertices and let $\ell = \log(2/\alpha) \cdot n$. Then, in the semi-random process on $n$ vertices, \Builder has a strategy which w.h.p. allows him to construct a copy of $T'$ within $\ell$ rounds.
\end{lemma} 

\begin{proof}
	Assume without loss of generality that $T'$ is a tree (otherwise simply replace $T'$ with a tree containing it). Let $t = (1 - \alpha) n$ and let $v_1, \ldots, v_t$ be an ordering of the vertices of $T'$ such that $T'[\{v_1, \ldots, v_i\}]$ is a tree for every $1 \leq i \leq t$. Throughout the process, \Builder maintains a partial function $\varphi$ which is initially empty. For every positive integer $i$, let $w_i$ denote the vertex \Builder is offered in the $i$th round. In the first round, \Builder connects $w_1$ to an arbitrary vertex $u$; he then sets $\varphi(v_1) = w_1$ and $\varphi(v_2) = u$. For every $i \geq 2$, \Builder plays the $i$th round as follows. Let $r$ denote the largest integer for which $\varphi(v_r)$ has already been defined. If $w_i \notin \{\varphi(v_1), \ldots, \varphi(v_r)\}$, then \Builder connects $w_i$ to $\varphi(v_j)$, where $j \leq r$ is the unique integer for which $\{v_{r+1}, v_j\} \in E(T')$; he then sets $\varphi(v_{r+1}) = w_i$. Otherwise, \Builder claims an arbitrary edge incident with $w_i$, which he does not consider to be part of the tree he is building (alternatively, \Builder skips this round).
	
	It is evident that, by following the proposed strategy, \Builder's graph contains a copy of $T'$ as soon as $t$ different vertices are offered. Hence, it suffices to prove that w.h.p. at least $t$ different vertices are offered during the first $\ell$ rounds.\footnote{This can be restated as saying that in the coupon collector's problem, $\ell = \log(2/\alpha)n$ rounds suffice in order to collect $t = (1-\alpha)n$ different coupons w.h.p.} For every $1 \leq j \leq n$, let $I_j$ be the indicator random variable for the event: ``vertex $j$ was not offered during the first $\ell$ rounds of the process''. Let $X = \sum_{j=1}^n I_j$; then 
	$$
	\mathbb{E}(X) = \sum_{j=1}^n \mathbb{E}(I_j) = n (1 - 1/n)^{\ell} \leq n \cdot e^{- \ell/n} = \alpha n/2. 
	$$ 
	Observe that changing any one coordinate in the sequence $(w_1, \ldots, w_{\ell})$ of random vertices, can change the value of $X$ by at most $1$. Hence, it follows by Lemma~\ref{lem:Azuma} with $c = 1$ and $\lambda = \alpha n/2$, that
	\begin{align*}
	\mathbb{P}[X \geq \alpha n] \leq \mathbb{P}\left[X \geq \mathbb{E}(X) + \alpha n/2 \right] \leq e^{- \frac{(\alpha n/2)^2}{2 \ell}} = 
	e^{- \frac{\alpha^2 n^2}{8 \log(2/\alpha) n}} = o(1) \; ,
	\end{align*}
	where the last equality holds by our assumption that $n \gg \alpha^{-2} \log(1/\alpha)$.
\end{proof}

\noindent
We are now in a position to prove Theorem~\ref{thm:trees}. 

\begin{proof}[Proof of Theorem~\ref{thm:trees}]
	Let $n$, $\Delta$, and $T$ be as in the statement of the theorem. 
	Assume first that $\Delta \geq n^{1/11}$. 
	In this case \Builder employs the strategy presented in the proof of Lemma \ref{lem:greedy_embedding}. It is easy to see that as soon as each of the $n$ vertices has been offered, \Builder's graph contains a copy of $T$. It is well-known (and easy to prove) that this will happen w.h.p. in $(1 + o(1)) n \log n = O(n \log \Delta)$ rounds. For the remainder of the proof we thus assume that $\Delta < n^{1/11}$.
	
	Set $\alpha = 10^{-8}\Delta^{-5}$, and note that $n \gg \alpha^{-2}\log(1/\alpha)$, as required by Lemma \ref{lem:greedy_embedding}. Since $T$ is a forest, there are at least $n/2$ vertices of $T$ whose degree is at most $2$. Fix a set $B$ of $\alpha n$ such vertices. Set $A := V(H) \setminus B$ and $T' := T[A]$. At the first stage of his strategy, \Builder invokes Lemma \ref{lem:greedy_embedding}, which enables him to construct a copy of $T'$ w.h.p. in $\log(2/\alpha)n = O(n \log \Delta)$ rounds.
		
	Let $\varphi : V(T) \rightarrow [n]$ be a bijection such that $\varphi|_{A}$ is an embedding of $T'$ into \Builder's graph. 
	Note that the conditions of Lemma \ref{lem:main} are satisfied with $D = 1$, since forests are $1$-degenerate, and since we made sure that all vertices in $B$ have degree at most $2$. In the second stage of his strategy, \Builder applies (the strategy given by) Lemma \ref{lem:main} in order to construct a copy of $T$ in his graph w.h.p. This requires $O(n)$ additional rounds. The total number of rounds is thus $O(n\log \Delta)$, as required. 
	\end{proof}

\subsection{A Lower Bound : Proof of Proposition \ref{prop::offlineTree}}

\begin{proof}[Proof of Proposition~\ref{prop::offlineTree}]
	Let $\Delta$ and $n \geq n_0(\Delta)$ be as in the statement of the proposition. Since \Builder clearly needs at least $n-1$ rounds in order to build a tree on $n$ vertices, we can assume that $\Delta$ is a sufficiently large constant. 
	We prove the proposition for the $n$-vertex forest $T$ consisting of $\lfloor \frac{n}{\Delta+1} \rfloor$ pairwise-disjoint $(\Delta+1)$-vertex stars, together with some (at most $\Delta$) isolated vertices (if needed). Let us denote the center of the $i$th star by $u_i$, and its leaves by $x_{i,1},\dots,x_{i,\Delta}$ ($1 \leq i \leq \lfloor \frac{n}{\Delta+1} \rfloor$).
	
	Let $w_1, w_2, \ldots$ denote the sequence of random vertices offered to \Builder, and let $m = 0.1 n \log \Delta$. Suppose that \Builder did manage to build a copy of $T$ within $m$ rounds, and let $\varphi : V \to [n]$ be a bijection such that $\{\varphi(u), \varphi(v)\}$ is an edge in \Builder's graph for every $\{u,v\} \in E(T)$. It is then evident that, for every 
	$1 \leq i \leq \lfloor \frac{n}{\Delta+1} \rfloor$, either $\varphi(u_i)$ appears at least $\sqrt{\Delta}$ times in $(w_1, w_2, \ldots, w_m)$, or at least 
	$\Delta - \sqrt{\Delta}$ of the elements of the set $\{\varphi(x_{i,j}) : 1 \leq j \leq \Delta\}$ appear at least once in $(w_1, w_2, \ldots, w_m)$. A straightforward calculation then shows that either at least 
	$\frac{1}{\sqrt{\Delta}} \cdot \lfloor \frac{n}{\Delta+1} \rfloor \geq \frac{n}{2\Delta^{3/2}}$ of the vertices in 
	$\{\varphi(u_i) : 1 \leq i \leq \lfloor \frac{n}{\Delta+1} \rfloor\}$ were offered at least $\sqrt{\Delta}$ times each, or all but at most 
	$$
	\Delta + 
	\Delta \cdot \frac{1}{\sqrt{\Delta}} \cdot \left\lfloor \frac{n}{\Delta+1} \right\rfloor + 
	\left\lfloor \frac{n}{\Delta+1} \right\rfloor \cdot (\sqrt{\Delta}+1) \leq \frac{3n}{\sqrt{\Delta}}
	$$
	of the $n$ vertices were offered at least once. In the inequality above we use the assumption that $n \geq n_0(\Delta)$ for some suitable $n_0(\Delta)$. 
	So in order to prove that w.h.p. \Builder needs more than $m$ rounds to build $T$, it suffices to show that w.h.p. there are more than $3n/\sqrt{\Delta}$ vertices $1 \leq i \leq n$ that do not appear in $(w_1, w_2, \ldots, w_m)$, and less than $\frac{n}{2 \Delta^{3/2}}$ vertices $1 \leq i \leq n$ which appear in $(w_1, w_2, \ldots, w_m)$ at least $\sqrt{\Delta}$ times. 
	
	Let $X$ be the random variable which counts the number of vertices $1 \leq i \leq n$ that do not appear in $(w_1, w_2, \ldots, w_m)$. Our goal is to show that w.h.p. $X > 3n/\sqrt{\Delta}$. We have 
	$$
	\mathbb{E}(X) = n(1 - 1/n)^m \geq n \cdot e^{- m/(n-1)} = 
	n \cdot e^{-0.1 \log \Delta \cdot \frac{n}{n-1}} = n \cdot \Delta^{-0.1-o(1)} > 4n/\sqrt{\Delta},
	$$ 
	where the first inequality follows from the fact that $(1 - 1/n)^{n-1} \geq 1/e$ for each $n \geq 1$, and the last inequality holds for sufficiently large $n$ and $\Delta$. We will use Lemma~\ref{lem:Azuma} to prove that w.h.p. $X$ is not much smaller than its expected value. Observe that changing any single coordinate in the sequence of random vertices $(w_1, w_2, \ldots, w_m)$ can change the value of $X$ by at most $1$. Therefore, applying Lemma~\ref{lem:Azuma} with parameters $c = 1$ and $\lambda = n/\sqrt{\Delta}$ yields 
	\begin{align*}
	\mathbb{P} \left[X \leq 3n/\sqrt{\Delta} \right] &\leq \mathbb{P} \left[X \leq \mathbb{E}(X) - n/\sqrt{\Delta} \right] \leq 
	e^{- \frac{(n/\sqrt{\Delta})^2}{2m}} \leq 
	e^{- \frac{(n/\sqrt{\Delta})^2}{n \log \Delta}}
	= e^{- \frac{n}{\Delta \log \Delta}} = o(1),
	\end{align*}
	where the last equality holds since $n$ is sufficiently large with respect to $\Delta$.
	
	Let $Z$ be the random variable which counts the number of vertices $1 \leq i \leq n$ that appear at least $\sqrt{\Delta}$ times in $(w_1, w_2, \ldots, w_m)$. Our goal is to prove that w.h.p. $Z < \frac{n}{2 \Delta^{3/2}}$. For every $1 \leq i \leq n$, let $Z_i$ be the random variable counting the number of times $i$ appears in $(w_1, w_2, \ldots, w_m)$. Then $Z_i \sim \Bin(m, 1/n)$, implying that $\mathbb{E}[Z_i] = m/n \leq 0.1 \log \Delta$. Applying Lemma~\ref{lem:Chernoff} with parameter $\lambda = 2 \log \Delta \leq \sqrt{\Delta} - 0.1 \log \Delta$, we obtain
	\begin{equation} \label{eq::veryLargeOutDeg}
	\mathbb{P}[Z_i \geq \sqrt{\Delta}] \leq 
	\mathbb{P}[Z_i \geq \mathbb{E}[Z_i] + \lambda] \leq 
	\exp\left( {- \frac{(2 \log \Delta)^2}{2(\mathbb{E}[Z_i] + \frac{2}{3}\log\Delta)}}\right) \leq 
	e^{-2\log\Delta} \leq 
	\frac{1}{4\Delta^{3/2}} \; ,
	\end{equation}
	where the last inequality holds for sufficiently large $\Delta$.  For every $1 \leq i \leq n$, let $I_i$ be the indicator random variable for the event $Z_i \geq \sqrt{\Delta}$; note that $Z = \sum_{i=1}^n I_i$. It follows by~\eqref{eq::veryLargeOutDeg} and by the linearity of expectation that $\mathbb{E}(Z) \leq \frac{n}{4\Delta^{3/2}}$. 
	Since changing any single coordinate in the sequence of random vertices $(w_1, w_2, \ldots, w_m)$ can change the value of $Z$ by at most $1$, applying Lemma~\ref{lem:Azuma} with parameters $c = 1$ and 
	$\lambda = \frac{n}{4\Delta^{3/2}}$ yields
	\begin{equation*}
	\mathbb{P} \left[Z \geq \frac{n}{2 \Delta^{3/2}} \right] \leq \mathbb{P} \left[Z \geq \mathbb{E}(Z) + \frac{n}{4\Delta^{3/2}} \right] \leq 
	e^{- \frac{n^2}{32\Delta^3 m}} 
	\leq e^{- \frac{n^2}{32 \Delta^3 n \log \Delta}} = o(1),
	\end{equation*}
	where the equality holds since $n$ is sufficiently large with respect to $\Delta$.
\end{proof}

\section{Non-Adaptive Strategies}\label{sec:non_adaptive}
In this section we prove Theorems \ref{prop::non-adaptive_isolated_vertices}, \ref{prop::non-adaptive_Hamilton_cycle} and \ref{prop::non-adaptive_clique_factor}.
\begin{proof}[Proof of Theorem~\ref{prop::non-adaptive_isolated_vertices}]
	Let $\mathcal{L} = \{L^w : w \in [n]\}$ be a family of lists as in the definition of a non-adaptive strategy. Recall that for each $w \in [n]$, the list $L^w$ is a permutation of $[n] \setminus \{w\}$. Our goal is to show that the strategy corresponding to $\mathcal{L}$ requires w.h.p. at least $\Omega(n\sqrt{\log n})$ rounds to make all $n$ vertices non-isolated. 
	
	Set $t = n\sqrt{\log n}/4$, and let $w_1,\dots,w_t$ be the first $t$ random vertices \Builder is offered. For every $v \in [n]$, let $t_v$ denote the number of appearances of $v$ in the sequence $(w_1, \ldots, w_t)$. Let 
	$U = \{v \in [n] : t_v > \sqrt{\log n}/2\}$ and let $W = \bigcup_{v \in U} \{L^v(i) : 1 \leq i \leq t_v\}$.
	Our main observation (which follows immediately from the definitions of $U$ and $W$) is that a vertex $u \in [n]$ will be left isolated after $t$ rounds, if all of the following conditions hold:
	\begin{description}
		\item [(1)] $u$ does not appear in $(w_1, \ldots, w_t)$;
		\item [(2)] none of the vertices $v$, for which $u$ is included among the first $\sqrt{\log n}/2$ elements of $L^v$, appear in $(w_1, \ldots, w_t)$;
		\item [(3)] $u \notin W$.
	\end{description}
	
	So in order to complete the proof of the theorem, it remains to prove that w.h.p. there exists a vertex $u \in [n]$ which satisfies Conditions (1), (2), and (3) above. 
	To this end, we will use a two-round exposure argument.
	Let $Z$ denote the set of vertices which do not appear in $(w_1, \ldots, w_t)$; clearly $Z \cap U = \emptyset$. In the following claim we collect some simple facts regarding the sets $U,Z$ and the integers $(t_v : v \in [n])$. 
	\begin{claim}\label{claim:non_adaptive}
		The following hold w.h.p.
		\begin{enumerate}
			\item[(a)] $|U| \leq e^{-\Omega(\sqrt{\log n})}n$.
			\item[(b)] $t_v < \log n$ for every $v \in [n]$. 
			\item[(c)] $|Z| \geq (1 - o(1))e^{-\sqrt{\log n}/4}n$.
		\end{enumerate}
	\end{claim}
	\begin{proof}
		We start with Item (a). Recall that for a given vertex $v \in [n]$, we have $t_v \sim \Bin(t,\frac{1}{n})$; hence, $\mathbb{E}[t_v] = \sqrt{\log n}/4$. Now, by Lemma \ref{lem:Chernoff} with $\lambda = \sqrt{\log n}/4$, we have 
		$$
		\mathbb{P}[v \in U] = \mathbb{P}[t_v > \sqrt{\log n}/2] =  
		\mathbb{P}[t_v > \mathbb{E}[t_v] + \sqrt{\log n}/4] \leq 
		e^{-\frac{\log n}{O(\sqrt{\log n})}} = 
		e^{-\Omega(\sqrt{\log n})} \; .
		$$ 
		It thus follows by Markov's inequality that w.h.p. $|U| \leq e^{-\Omega(\sqrt{\log n})}n$. 
		
		We now prove Item (b). Observe that for every $v \in [n]$ we have
		$$
		\mathbb{P}[t_v \geq \log n] \leq 
		\binom{t}{\log n}\left( \frac{1}{n} \right)^{\log n} \leq \left( \frac{et}{n\log n} \right)^{\log n} \leq \left( \frac{1}{\sqrt{\log n}} \right)^{\log n} = o(1/n).
		$$
		A union bound over $[n]$ then shows that w.h.p. $t_v < \log n$ for every $v \in [n]$.
		
		Finally, we prove Item (c). For each $v \in [n]$, the probability that $v \in Z$ is $(1 - 1/n)^t = (1 - o(1))e^{-\sqrt{\log n}/4}$. Therefore, $\mathbb{E}[|Z|] = (1 - o(1))e^{-\sqrt{\log n}/4}n$. To show that $|Z|$ is concentrated around its expected value, observe that changing any single coordinate in the sequence $(w_1,\dots,w_t)$ of random vertices, can change the value of $|Z|$ by at most $1$. Hence, by Lemma~\ref{lem:Azuma} with $c=1$ and (say) $\lambda = n^{2/3}$, we have 
		$$
		\mathbb{P}\left[ |Z| \leq \mathbb{E}[|Z|] - n^{2/3} \right] \leq 
		e^{-\frac{n^{4/3}}{2t}} = e^{-\frac{n^{4/3}}{O(n\sqrt{\log n})}} = o(1).
		$$ 
		We conclude that w.h.p. $|Z| \geq (1 - o(1))e^{-\sqrt{\log n}/4}n$. 
	\end{proof}
	
	From now on we condition on the events stated in Items (a)-(c) of Claim \ref{claim:non_adaptive} (which hold w.h.p. by that claim). Items (a) and (b) imply that $|W| \leq |U| \log n \leq \log n \cdot e^{-\Omega(\sqrt{\log n})}n = o(n)$. 
	Observe that conditioning on their sizes, $U,Z$ are uniformly distributed among all pairs of disjoint subsets of $[n]$ of the corresponding sizes. 
	From this point on we condition on $U$, which in turn determines $W$ (or, more precisely, the set $\bigcup_{v \in U} \{L^v(i) : 1 \leq i \leq \log n\}$, which contains $W$ and has size $o(n)$).  
	
	For each $u \in [n]$, let $A_u$ be the set of all vertices $v \in [n] \setminus \{u\}$ such that $u$ is included among the first $\sqrt{\log n}/2$ elements of $L^v$.
	Let $V_0$ be the set of all $u \in [n]$ satisfying $|A_u| \leq \sqrt{\log n}$. Since the union (as a multiset) of the first $\sqrt{\log n}/2$ elements in all lists has size $n \sqrt{\log n}/2$ altogether, we deduce that $|V_0| \geq n/2$; hence $|V_0 \setminus W| \geq (1/2 - o(1))n$. 
	
	Now expose $Z$, conditioning on its size. 
	Observe that if $u \in [n] \setminus W$ is such that $A_u \cup \{u\} \subseteq Z$, then $u$ satisfies Conditions (1), (2) and (3). So from now on our goal is to show that w.h.p. there exists a vertex $u \in [n] \setminus W$ for which 
	$A_u \cup \{u\} \subseteq Z$. Since each $u \in [n]$ belongs to at most 
	$\sqrt{\log n}/2 + 1$ of the sets $\left\{ A_v \cup \{v\}  : v \in [n] \right\}$, and since $|A_u \cup \{u\}| \leq \sqrt{\log n} + 1$ for each $u \in V_0$, one can find a collection $B_1,\dots,B_{s'}$ of 
	$$
	s' \geq 
	\left\lfloor 
	\frac{|V_0 \setminus W|}{(\sqrt{\log n}/2 + 1)(\sqrt{\log n} + 1)+1} \right\rfloor
	= \Omega\left( \frac{n}{\log n} \right)
	$$ 
	pairwise-disjoint sets among the sets 
	$\{A_u \cup \{u\} : u \in V_0 \setminus W\}$. 
	Since $B_1,\dots,B_{s'}$ are pairwise-disjoint, at least $s := s' - |U| \geq \Omega(n/\log n) - e^{-\Omega(\sqrt{\log n})}n = \Omega(n/\log n)$ of these sets are contained in $[n] \setminus U$. So suppose, without loss of generality, that $B_1,\dots,B_s \subseteq [n] \setminus U$, and let us show that w.h.p., there is $1 \leq i \leq s$ such that $B_i \subseteq Z$. 
	To this end, we will couple $Z$ with a binomial random set of slightly smaller size. 
	Recall that we are conditioning on the size of $Z$ and on the event  
	$|Z| \geq (1 - o(1))e^{-\sqrt{\log n}/4}n$, which occurs w.h.p. by Claim~\ref{claim:non_adaptive}. We denote $z = |Z|$, recalling that (under this conditioning), $Z$ is distributed uniformly among all subsets of $[n] \setminus U$ of size $z$. We generate $Z$ by performing the following experiment: set $p = \frac{z}{2n}$, and let $R$ be a random subset of $[n] \setminus U$, obtained by independently including each element of $[n] \setminus U$ with probability $p$. If $|R| \leq z$, then we uniformly choose a set $Z' \subseteq [n] \setminus U$ of size $z$ which contains $R$. It is easy to see that, conditioned on $|R| \leq z$, the set $Z'$ is distributed uniformly among all subsets of $[n] \setminus U$ of size $z$. Hence (conditioned on $|R| \leq z$), $Z'$ has the same distribution as $Z$ (conditioned on $|Z| = z$). Note that $|R|$ is stochastically dominated by $\Bin(n,\frac{z}{2n})$, so by Lemma \ref{lem:Chernoff} with $\lambda = \frac{z}{2}$ we have
	$$
	\mathbb{P}[|R| > z] \leq 
	\mathbb{P}\left[ |R| \geq \mathbb{E}[|R|] + \frac{z}{2}  \right] \leq 
	e^{-\frac{(z/2)^2}{2(\mathbb{E}[|R|] + z/6)}} \leq e^{-\Omega(z)} = o(1).
	$$
	Since the sets $B_1, \ldots, B_s$ are pairwise-disjoint and of size at most $\sqrt{\log n} + 1$ each, the probability that $R$ contains none of these sets is at most
	\begin{align*}
	\left( 1 - p^{\sqrt{\log n} + 1} \right)^{s} &\leq 
	\exp\left( {-p^{\sqrt{\log n} + 1} \cdot s} \right) = 
	\exp\left( {-\left( \frac{z}{2n} \right)^{\sqrt{\log n} + 1} \cdot s} \right) \\
	&\leq \exp \left(- {e^{- \log n/2}} \cdot s \right) = e^{- s/\sqrt{n}} = e^{-\Omega(\sqrt{n}/\log n)} = o(1),
	\end{align*}
	where in the second inequality we used the assumption that 
	$z \geq (1 - o(1))e^{-\sqrt{\log n}/4}n$. We conclude that w.h.p. there will be some $1 \leq i \leq s$ such that $B_i \subseteq Z$, as required.  
\end{proof}

%
%

\begin{proof}[Proof of Theorem \ref{prop::non-adaptive_Hamilton_cycle}]
	Setting $k = n/\sqrt{\log n}$, partition $[n]$ into sets $V_1,\dots,V_k$, each of size either $\lfloor n/k \rfloor$ or $\lceil n/k \rceil$. 
	For each $1 \leq i \leq k$ and $v \in V_i$, \Builder sets the adjacency list $L^v$ so that first appear all the vertices of $V_i \setminus \{v\}$, then all the vertices of $V_{i+1}$ (where $i+1$ is taken modulo $k$), and finally all other vertices. In each of the three ``segments", the inner order among the vertices is arbitrary.

	Now, set $t = 8n \sqrt{\log n}$, and let $(w_1, \ldots, w_t)$ be the first $t$ random vertices \Builder is offered. Let $W$ be the set of vertices appearing at most $\lceil n/k \rceil$ times in $(w_1, \ldots, w_t)$. Let $\mathcal{A}$ be the event that $|V_i \cap W| \leq |V_i|/2 - 1$ for each $1 \leq i \leq k$. We will show that $\mathcal{A}$ happens w.h.p., and that if $\mathcal{A}$ happens then after $t$ rounds, \Builder's graph contains a Hamilton cycle. 
	
	We start by estimating $\mathbb{P}[\mathcal{A}]$. Fixing $1 \leq i \leq k$, note that if $|V_i \cap W| \geq |V_i|/2$ then there is a set $U \subseteq V_i$ of size $|U| = |V_i|/2 = (1+o(1))\sqrt{\log n}/2$, such that $X := |\{1 \leq j \leq t : w_j \in U\}| \leq |U| \cdot \lceil n/k \rceil \leq \log n$. Note that $X$ has the distribution $\Bin(t, |U|/n)$. By Lemma \ref{lem:Chernoff} with $\lambda = (3+o(1))\log n$, and by our choice of $t$, we get
	\begin{align*}
	\mathbb{P}[X \leq \log n] &\leq 
	\mathbb{P} \left[\Bin\left(8n \sqrt{\log n} \; , \; \frac{(1+o(1))\sqrt{\log n}}{2n} \right) \leq \log n \right]
	\leq 
	\exp\left( -{\frac{(1+o(1))9\log^2n}{8\log n}} \right) = 
	o(n^{-1.1}) \; . 
	\end{align*}
	By taking the union bound over all $k = o(n)$ indices $1 \leq i \leq k$, and over all at most $2^{|V_i|} = 2^{(1+o(1))\sqrt{\log n}} = n^{o(1)}$ choices of $U \subseteq V_i$, we obtain $\mathbb{P}[\mathcal{A}] = 1 - o(1)$. 
	
	Suppose now that $\mathcal{A}$ happened. Then for each $1 \leq i \leq k$, each of the at least $|V_i|/2 + 1$ vertices $v \in V_i \setminus W$ has been connected to all vertices in $V_i$, and to at least $2$ vertices in $V_{i+1}$ (indeed, this is due to our choice of the lists $L_v$, and the definition of $\mathcal{A}$). It follows that the minimum degree inside $V_i$ is at least $|V_i|/2 + 1$ (for each $1 \leq i \leq k$), and that we can choose distinct vertices $x_i,y_i \in V_i$ such that $y_i$ is connected to $x_{i+1}$ for each $1 \leq i \leq k$ (with indices taken modulo $k$). 
	
	Recall that a graph is called {\em Hamilton-connected} if for each pair of distinct vertices $u,v$, there is a Hamilton path whose endpoints are $u$ and $v$. It follows from a classical result of Ore \cite{Ore} that any $m$-vertex graph with minimum degree at least $\frac{m + 1}{2}$ is Hamilton-connected. Let $G$ denote \Builder's graph immediately after $t$ rounds of the process. By the above result of Ore, $G[V_i]$ is Hamilton-connected for each $1 \leq i \leq k$. So fix, for each $1 \leq i \leq k$, a Hamilton path $P_i$ in $G[V_i]$, whose endpoints are $x_i$ and $y_i$. Now it is easy to see that $P_1,\{y_1,x_2\},P_2,\{y_2,x_3\},\dots,\{y_{k-1},x_k\},P_k,\{y_k,x_1\}$ is a Hamilton cycle in $G$, as required. 
\end{proof}

\begin{proof}[Proof of Theorem \ref{prop::non-adaptive_clique_factor}]
	The proof is somewhat similar to the proof of Theorem \ref{prop::non-adaptive_Hamilton_cycle}, and so we only give a rough sketch. 
	Partition $[n]$ into $k = n/\sqrt{\log n}$ parts $V_1, \ldots, V_k$ whose sizes are all divisible by $r$ and are as close to each other as possible. Let $s_{\min} = \min \{|V_i| : 1 \leq i \leq k\}$ and let $s_{\max} = \max \{|V_i| : 1 \leq i \leq k\}$; observe that $s_{\min}, s_{\max} = (1\pm o(1)) \sqrt{\log n}$. Then, for every $1 \leq i \leq k$ and every $v \in V_i$, \Builder sets the adjacency list $L^v$ so that first appear all the vertices of $V_i \setminus \{v\}$ (in an arbitrary order), and then all other vertices (in an arbitrary order).
	
	Now, set $t = C n \sqrt{\log n}$ (where $C = C(r)$ will be chosen later), and let $(w_1, \ldots, w_t)$ be the first $t$ random vertices offered to \Builder. Let $W$ be the set of vertices appearing at most $s_{\max} - 2$ times in $(w_1, \ldots, w_t)$. Observe that if $|W \cap V_i| < s_{\min}/r$, then the resulting induced subgraph of \Builder $G[V_i]$ has minimum degree at least $\left(1 - 1/r \right) |V_i|$ and thus admits a $K_r$-factor by the Hajnal-Szemer\'edi Theorem~\cite{HS}. If this happens for every $1 \leq i \leq k$, then the union over $1 \leq i \leq k$ of these $K_r$-factors obviously forms a $K_r$-factor of $G$. It remains to prove that w.h.p. $|W \cap V_i| < s_{\min}/r$ holds for every $1 \leq i \leq k$. Fix some $1 \leq i \leq k$. Then
	\begin{align*}
	\mathbb{P}[|W \cap V_i| \geq s_{\min}/r] &\leq \binom{s_{\max}}{s_{\min}/r} \cdot \mathbb{P} \left[\Bin(t, s_{\min}/(r n)) \leq (s_{\max} - 2) \cdot s_{\min}/r \right] \\ 
	&\leq (3 r)^{s_{\min}/r} \cdot \exp \left\{- \frac{r n}{2 t s_{\min}} \cdot \frac{C^2 \cdot s_{\min}^4}{4 r^2} \right\} 
	\\ &=  
	(3r)^{s_{\min}/r} \cdot 
	\exp\left\{ -\frac{C \cdot s_{\min}^3}{8r\sqrt{\log n}} \right\}
	\\ &\leq 
	(3r)^{\sqrt{\log n}} \cdot 
	\exp\left\{ -\frac{C \cdot (1-o(1))\log n}{8r} \right\}
	\\
	&\leq \exp \left\{- \frac{C \log n}{9 r} \right\} = o(1/k),
	\end{align*}
	where the second inequality holds by Lemma~\ref{lem:Chernoff} with $\lambda = \frac{C}{2} \cdot \frac{s_{\min}^2}{r}$, and the equality holds if, say, $C = 9r$. We also assumed throughout that $n$ is large enough with respect to $r$. A union bound over all $1 \leq i \leq k$ then shows that w.h.p. $|W \cap V_i| < s_{\min}/r$ holds for every $1 \leq i \leq k$, as required. 
\end{proof}

\bibliographystyle{plain}
\bibliography{semi_random}

%
%
%
%
\end{document}